\documentclass[10pt,a4paper]{amsart}

\usepackage{fancyhdr,amsmath,amssymb,latexsym,verbatim}
\usepackage{tikz,rotating,tabularx,setspace,url,xcolor}

\usepackage[total={6.5in,9.5in},centering,includefoot,includehead]{geometry}

\usepackage{amssymb,amsthm,amsfonts,amsmath}

\usepackage[cmtip,all]{xy}
  \newdir{ >}{{}*!/-9pt/@{>}}
  \tikzset{vertex/.style={circle,draw,fill,scale=.35}}
\tikzset{op/.style={circle,fill,scale=.35}}
\tikzset{cl/.style={circle,fill,scale=.35}}
\tikzset{loopdown/.style={loop below,min distance=8mm,in=310,out=230,looseness=25}}
\tikzset{loopup/.style={loop above,min distance=8mm,in=130,out=50,looseness=25}}

\usepackage{faktor}

\newcommand{\R}{\mathbb{R}}
\newcommand{\Q}{\mathbb{Q}}
\newcommand{\Z}{\mathbb{Z}}
\newcommand{\C}{\mathbb{C}}
\newcommand{\T}{\mathbb{T}}
\newcommand{\G}{\mathcal{G}}
\newcommand{\lie}{\mathfrak{g}}
\newcommand{\lien}{\mathfrak{n}}

\newcommand{\liem}{\mathfrak{m}}

\newcommand{\n}{\mbox{$\mathfrak{n}$}}

\newcommand{\oprec}{\overline{\prec}}

\DeclareMathOperator{\Aff}{Aff}

\DeclareMathOperator{\GL}{GL}

\DeclareMathOperator{\Aut}{Aut}

\DeclareMathOperator{\End}{End}

\DeclareMathOperator{\Perm}{Perm}

\newcommand{\rvline}{\hspace*{-\arraycolsep}\vline\hspace*{-\arraycolsep}}

\newtheorem{theorem}{Theorem}[section]
\newtheorem{lemma}[theorem]{Lemma}
\newtheorem{proposition}[theorem]{Proposition}
\newtheorem{corollary}[theorem]{Corollary}

\newtheorem{QN}{Question}
\newtheorem*{Con}{Conjecture}
\newtheorem*{Prop*}{Proposition}
\newtheorem*{Lem*}{Lemma}

\theoremstyle{definition}
\newtheorem{definition}[theorem]{Definition}
\newtheorem{example}[theorem]{Example}

\makeatletter
\newtheorem*{rep@theorem}{\rep@title}
\newcommand{\newreptheorem}[2]{%
	\newenvironment{rep#1}[1]{%
		\def\rep@title{#2 \ref{##1}}%
		\begin{rep@theorem}}%
		{\end{rep@theorem}}}
\makeatother

\newreptheorem{theorem}{Theorem}

\theoremstyle{remark}
\newtheorem{remark}[theorem]{Remark}

\numberwithin{equation}{section}

\newcommand{\compcent}[1]{\vcenter{\hbox{$#1\circ$}}}

\newcommand{\comp}{\mathbin{\mathchoice
		{\compcent\scriptstyle}{\compcent\scriptstyle}
		{\compcent\scriptscriptstyle}{\compcent\scriptscriptstyle}}}
	
	\newcommand\restr[2]{{
			\left.\kern-\nulldelimiterspace 
			#1 
			\vphantom{\big|} 
			\right|_{#2} 
	}}

\begin{document}

\title[Anosov diffeomorphisms on infra-nilmanifolds associated to graphs]{\bf Anosov diffeomorphisms on infra-nilmanifolds \\ associated to graphs}


\author{Jonas Der\'e}

\date{}
\address{KU Leuven Kulak, E. Sabbelaan 53, 8500 Kortrijk, Belgium}
\email{jonas.dere@kuleuven.be}
\thanks{The first author was supported by a postdoctoral fellowship of the Research Foundation -- Flanders (FWO)}

\author{Meera Mainkar}
\address{Department of Mathematics, Pearce Hall, Central Michigan University, Mt. Pleasant, MI 48858, USA} 
\email{maink1m@cmich.edu}


\begin{abstract}
Anosov diffeomorphisms on closed Riemannian manifolds are a type of dynamical systems exhibiting uniform hyperbolic behavior. Therefore their properties are intensively studied, including which spaces allow such a diffeomorphism. It is conjectured that any closed manifold admitting an Anosov diffeomorphism is homeomorphic to an infra-nilmanifold, i.e.~a compact quotient of a $1$-connected nilpotent Lie group by a discrete group of isometries. This conjecture motivates the problem of describing which infra-nilmanifolds admit an Anosov diffeomorphism.

So far, most research was focused on the restricted class of nilmanifolds, which are quotients of $1$-connected nilpotent Lie groups by uniform lattices. For example, Dani and Mainkar studied this question for the nilmanifolds associated to graphs, which form the natural generalization of nilmanifolds modeled on free nilpotent Lie groups. This paper further generalizes their work to the full class of infra-nilmanifolds associated to graphs, leading to a necessary and sufficient condition depending only on the induced action of the holonomy group on the defining graph. As an application, we construct families of infra-nilmanifolds with cyclic holonomy groups admitting an Anosov diffeomorphism, starting from faithful actions of the holonomy group on simple graphs.
\end{abstract}

\maketitle

\section{Introduction}

A diffeomorphism $f: M \to M$ on a closed Riemannian manifold $M$ is called \emph{Anosov} if the tangent bundle $TM$ has a continuous splitting $TM = E^u \oplus E^s$ preserved by $Df: TM \to TM$ such that $Df$ exponentially expands $E^u$ and exponentially contracts $E^s$ with respect to the Riemannian metric. This property does not depend on the choice of metric on $M$ and hence we will just talk about Anosov diffeomorphisms on a closed manifold. The first example of an Anosov diffeomorphism was \emph{Arnolds' cat map}, which is the map induced by the matrix $\begin{pmatrix} 2 & 1\\ 1 & 1 \end{pmatrix}$ on the $2$-torus $\T^2 = \faktor{\R^2}{\Z^2}$. Similarly it is easy to find for every $n \geq 2$ matrices in $\GL(n,\Z)$ which induce Anosov diffeomorphisms on the torus $\T^n = \faktor{\R^n}{\Z^n}$, whereas the circle $S^1$ does not admit such a map. Note that the $n$-dimensional torus is exactly the nilmanifold modeled on the abelian Lie group $\R^n$.

In his seminal paper \cite{smal67-1}, S. Smale gave the first example of a non-toral Anosov diffeomorphism and raised the question of classifying the closed manifolds $M$ admitting such an Anosov diffeomorphism. After more than $50$ years, the only known examples are on spaces homeomorphic to infra-nilmanifolds, which we will introduce in full detail in Section \ref{sec:IN}. In short, an infra-nilmanifold is a compact quotient of a simply connected nilpotent Lie group $N$ by a discrete subgroup of isometries. Nowadays, it is conjectured that every manifold admiting an Anosov diffeomorphism is in fact homeomorphic to an infra-nilmanifold with an Anosov diffeomorphism, but unfortunately there is no recent progress towards a proof. Hence most research focuses on the problem of describing the infra-nilmanifolds admitting an Anosov diffeomorphism. 

Even for the restricted class of nilmanifolds, this is a hard question, leading to a variety of techniques for constructing and classifying nilmanifolds admitting such diffeomorphisms. For example, the papers \cite{laur03-1,payn09-1,dere13-1} give different methodes for constructing lattices in nilpotent Lie groups, leading to Anosov diffeomorphisms satisfying additional properties. Some of these constructions are general enough to give a complete list of possibilities, e.g.~\cite{lw09-1} gives a classification of Anosov diffeomorphisms on nilmanifolds of dimension $\leq 8$, which was slightly corrected by the second author in \cite{dere13-1}. 

One particular paper in this direction is \cite{dm05-1} which gives a full classification in the special case of nilmanifolds associated to graphs, generalizing the result of \cite{dani78-1} which treated the free nilpotent Lie groups. In the general case of infra-nilmanifolds, a lot less is known, and the only full characterization of Anosov diffeomorphisms is in the case of infra-nilmanifolds modeled on free nilpotent Lie groups in \cite{dv11-1,dd13-1}, which extended the classical result of Porteous for flat manifolds \cite{port72-1}. In this paper, we generalize the latter result to the class of infra-nilmanifolds modeled on Lie groups associated to graphs.

To state the main result, we first introduce some notations. Every infra-nilmanifold induces a unique rational Lie algebra $\lien^\Q$ and a representation of a finite subgroup $\rho: H \to \Aut(\lien^\Q)$ of automorphisms, which is called the rational holonomy representation. In the case of infra-nilmanifolds associated to a graph $\G$, this rational holonomy representation induces an action on the coherent components of the graph $\G$ and representations $\rho_i: H_i \to \GL(V_{\lambda_i})$ for some coherent components $\lambda_i$. The subgroups $H_i \le H$ are the stabilizers of the coherent components $\lambda_i$ for the action of $H$. We show that these representations completely determine the existence of an Anosov diffeomorphism.
\begin{theorem}\label{maintheoremintro}
Let $\Gamma \backslash  N_\G$ be an infra-nilmanifold associated to a graph $\G$ with rational holonomy representation $\rho: H \to \Aut(\lien^\Q_\G)$. If $\rho_i: H_i \to \GL(V_{\lambda_i})$ are the induced representations on the coherent components, then the following are equivalent.

\begin{center}
	The infra-nilmanifold $\Gamma \backslash N_\G$ admits an Anosov diffeomorphism. \\
	$\Updownarrow$ \\
	For every $1 \leq i \leq k$, every $\Q$-irreducible component of $\rho_i$ that occurs with multiplicity $m$\\ splits in more than $\frac{c(H \cdot \lambda_i)}{m}$ components over $\R$.
\end{center}
\end{theorem}
\noindent The numbers $c(H \cdot \lambda_i)$  are either $1$ or $2$ and depend only on the orbit $H \cdot \lambda_i$ of the coherent component $\lambda_i$ under the action of $H$. A more detailed explanation of the notation in this theorem will follow in Section \ref{sec:mainresult}. 

In the special case where $H$ acts faithfully on the coherent components, each group $H_i$ is trivial and the criterion only depends on the number of elements in the component. We use this observation to give families of examples starting from faithful actions on finite graphs, leading to several families of infra-nilmanifolds modeled on graphs admitting an Anosov diffeomorphism.

We start by giving some general results about Anosov diffeomorphisms on infra-nilmanifolds in Section \ref{sec:AnosovonIN}, where the crucial ingredient is the characterization of \cite{dv11-1} in terms of rational Lie algebras and the rational holonomy representation. Section \ref{sec:fullauto} gives then a full description of the automorphism group of Lie algebras associated to graphs, including an induced action of finite subgroups on the coherent components. Next, we apply these results to get the main result in Section \ref{sec:mainresult} and present the construction for examples in Section \ref{sec:examples}. Finally, we state some open questions in Section \ref{sec:openQ}, related to other rational forms of Lie algebras associated to graphs.

\section{Anosov diffeomorphisms on infra-nilmanifolds}
\label{sec:AnosovonIN}

In this section, we introduce the necessary definitions about infra-nilmanifolds and recall the characterization in \cite{dv11-1} of infra-nilmanifolds admitting an Anosov diffeomorphism depending on the rational holonomy representation. Just as in the case of free nilpotent Lie groups, this will be the starting point for analyzing the existence of Anosov diffeomorphisms on Lie groups associated to graphs.

\subsection{Infra-nilmanifolds}
\label{sec:IN}

Let $N$ be a connected and simply connected (hereinafter $1$-connected) nilpotent Lie group and let $\Aut(N)$  denoted the group of Lie group automorphisms of $N$. The {\em affine group} $\Aff(N)$ is defined as the semi-direct product $N \rtimes \Aut(N)$ and acts on the Lie group $N$ on the left as follows:
\begin{eqnarray*}
	 (n, \phi) \cdot x = n \phi(x) \,\,\mbox{ for all } (n,\phi) \in \Aff(N), \hspace{1mm} x \in N.
\end{eqnarray*}

Let $C$ be a compact subgroup of $\Aut(N)$ and let $\Gamma$ be a discrete, torsion-free subgroup of $N \rtimes C$ such that the quotient  $\Gamma \backslash N$ is compact. We note that such a $\Gamma$ is known as an {\em almost-Bieberbach group} and since $C$ is compact, the group $\Gamma$ can be seen as a subgroup of isometries on $N$ for a suitable Riemannian metric. The quotient space $\Gamma \backslash N$ is a closed manifold and is called an {\em infra-nilmanifold modeled on the Lie group $N$}. If $C$ is trivial, then $\Gamma \le N$ is a uniform lattice and in this case the manifold $\Gamma \backslash N$ is called a nilmanifold. 

It is known that the normal subgroup $M = \Gamma \cap N$ of $\Gamma$ is a uniform lattice in $N$ and that the quotient $\faktor{\Gamma}{M} $ is finite by \cite{ausl60-1}, where we identify $N$ with the subgroup of left translations $N \times \{1\}$ of  $N \rtimes C$. This implies that the infra-nilmanifold  $\Gamma \backslash N$ is finitely covered by the nilmanifold $M\backslash N$ with $H = \faktor{\Gamma}{M}$ as the group of covering transformations. Let $p: \Gamma \to C$ denote the natural projection on the second component, then $H = p(\Gamma) \approx \faktor{\Gamma}{M}$ is a finite group which is known as the {\em holonomy group} of $\Gamma$. An infra-nilmanifold is a nilmanifold if and only if the holonomy group is trivial. The group $\Gamma$ fits in the following short exact sequence:
  \begin{align}\label{seq}
\xymatrix{1 \ar[r] & M \ar[r]  & \Gamma \ar[r]^p & H \ar[r] & 1}.
\end{align}
Without loss of generality, one can assume that $C = p(\Gamma) = H$. The infra-nilmanifolds modeled on the additive group $\R^m$ are exactly the closed flat Riemannian manifolds. 

Let $\alpha =(n, \phi) \in \Aff(N)$ be an affine transformation such that $\alpha \Gamma \alpha^{-1} = \Gamma$. Then $\alpha$ induces a diffeomorphism $\bar{\alpha}$ on the infra-nilmanifold $\Gamma \backslash N$, which is defined by 
\begin{align*}\bar{\alpha}: \Gamma \backslash N &\to \Gamma \backslash N \\ \Gamma x &\mapsto \bar{\alpha}(\Gamma x) := \Gamma \left( \alpha \cdot x \right) = \Gamma n\phi(x).\end{align*}  A diffeomorphism of an infra-nilmanifold  as above is called an {\em affine infra-nilmanifold automorphism}. The map $\bar{\alpha}$ is an Anosov diffeomorphism if and only if the linear part $\phi$ of $\alpha$ is hyperbolic, i.e.\ all the eigenvalues of $\phi$ are of absolute value different from $1$. The eigenvalues of $\phi$ are defined as the eigenvalues of the corresponding automorphism on the Lie algebra corresponding to $N$, so the eigenvalues of the differential of $\phi$ at the identity.

The only known examples of Anosov diffeomorphisms are those which are  topologically conjugate to affine infra-nilmanifold automorphisms. In fact, it is conjectured that every Anosov diffeomorphism is topologically conjugate to an affine infra-nilmanifold automorphism, see \cite{deki11-1,smal67-1}. Therefore  an important question is to study which infra-nilmanifolds admit an affine hyperbolic infra-nilmanifold automorphism. Note that an infra-nilmanifold admits an Anosov diffeomorphism if and only if it admits a hyperbolic affine infra-nilmanifold automorphism by \cite{deki11-1}.

\subsection{Rational holonomy representation}
For infra-nilmanifolds modeled on the abelian Lie group $\R^m$, the short exact sequence \eqref{seq} splits and thus gives rise to a natural representation of the holonomy group $$\rho: H \to \Aut(\Z^m) = \GL(m,\Z).$$ For general infra-nilmanifolds, this is not the case and therefore we introduce the rational holonomy representation associated to an infra-nilmanifold $\Gamma$, see \cite{dv11-1}.

Let $\n$ denote the nilpotent Lie algebra associated to the nilpotent Lie group $N$. It is known that if $N$ is $1$-connected, then  the exponential map $\exp: \n \to N$ is a diffeomorphism, see for example \cite{ragh72-1}. Let $\log : N \to \n$ denote the inverse of the map $\exp$. As before we take $M = \Gamma \cap N$ as uniform lattice in $N$ and let $\n^{\Q}$ denote the $\Q$-span of $\log(M)$, which is a rational Lie algebra by \cite{sega83-1}. We define $M^{\Q}$ as $\exp(\n^{\Q})$ and this group is known as the {\em rational Mal'cev completion or radicable hull of $M$}. It is the unique torsion-free radicable nilpotent group containing $M$, such that every element of $M^\Q$ has a positive power lying in $M$, see \cite{sega83-1}. The rational subalgebra $\lien^\Q$ is called a {\em rational form} of the real Lie algebra $\lien$.
 
To obtain the rational holonomy representation, one embeds the lattice $M$ into its rational Mal'cev completion $M^\Q$. Since every automorphism of $M$ uniquely extends to an automorphism of $M^\Q$, there exists a group $\Gamma^\Q$ which fits in the following commutative diagram:

$$ \xymatrix{
	1 \ar[r] & M \ar[r] \ar@{ >->}[d]& \Gamma \ar[r] \ar@{ >->}[d] & H \ar[r] \ar@{=}[d]& 1 \\
	1 \ar[r] & M^\Q \ar[r] & \Gamma^\Q \ar[r] & H \ar[r] & 1,}$$ where the bottom exact sequence splits, see \cite[Section 3.1]{deki96-1} for the details. By fixing a splitting morphism $s: H \to \Gamma^\Q$, we define the rational holonomy representation $\rho: H \to \Aut(M^\Q)$ by $$\rho(h)(n) = s(h) n s(h)^{-1}.$$ Equivalently, since $\Aut(M^\Q) \approx \Aut(\lien^\Q)$ under the exponential map, we can consider the representation as $\rho: H \to \Aut(\lien^\Q)$ into the automorphisms of the corresponding Lie algebra. We will not distinguish between these representations and call them both the rational holonomy representation of the almost-Bieberbach group $\Gamma$. The map $\rho$ is always injective, so the rational holonomy representation is faithful. Often we will identify the holonomy group with its image under the rational holonomy representation.

The rational holonomy representation does depend on the choice of $s$, but this dependence is not relevant in the study of Anosov diffeomorphisms. In fact, by the work of \cite[Theorem A]{dv08-1} the rational holonomy representation contains all information about the existence of Anosov diffeomorphisms.  \begin{theorem}
	\label{charkk}
Let $\Gamma \backslash N$ be an infra-nilmanifold with rational holonomy representation $\rho: H \to \Aut(M^\Q)$. Then $\Gamma \backslash N$ admits an Anosov diffeomorphism if and only if there exists an integer-like hyperbolic automorphism of $M^\Q$ which commutes with every element of $\rho(H)$.
\end{theorem}
\noindent Here, {\em integer-like} means that the characteristic polynomial of the corresponding automorphism on the Lie algebra $\lien^\Q$ has coefficients in $\Z$ and constant term $\pm 1$. An automorphism $\varphi \in \Aut(M^\Q)$ satisfying the conditions of the theorem is called {\em Anosov} and every rational Lie algebra admitting an Anosov automorphism is called an {\em Anosov Lie algebra}. Therefore studying Anosov diffeomorphisms on infra-nilmanifolds is equivalent to studying Anosov automorphisms of rational Lie algebras commuting with certain finite subgroups of the Lie algebra. In this paper, we restrict ourselves to the case of Lie algebras associated to graphs, for which we give a full description of the automorphism group in the next section.

\section{The automorphism group of Lie groups associated to graphs}
\label{sec:fullauto}

In this section we describe the automorphism group of Lie groups, or equivalently Lie algebras, associated to graphs, improving the result of \cite{dm05-1} which describes only the connected component of the identity element. The methods lead to an induced action of the holonomy group on the coherent components of the graph, which is crucial for our final results. We start by recalling the construction of a $2$-step nilpotent Lie algebra $\lien_\G^K$ associated to a finite simple graph $\G$ as in \cite{dm05-1}) over any field $K$ of characteristic $0$.

\subsection{Lie algebra associated to graphs}\label{subsec:graphs}

Let $\mathcal{G} = (S, E)$ denote a finite simple graph where $S$ is the set of vertices and $E$ is the set of edges.  We associate with $\mathcal{G}$ a 2-step nilpotent Lie algebra $\lien_\mathcal{G}^K$ over any field  $K$ of characteristic $0$ in the following way. The underlying vector space of $\lien_\mathcal{G}^K$ is $V \oplus  W$ where $V$ is the $K$-vector space with basis $S$ and $W$ is the subspace of $\bigwedge^2 V$ spanned by  $\{\alpha \wedge \beta \mid \alpha \beta \in E\}$. The Lie bracket structure on $\lien_\mathcal{G}^K$ is  given by the following  relations:
\begin{enumerate}
  \item  $[\alpha, \beta] = \alpha \wedge \beta$ for all $\alpha, \beta \in S$ with $\alpha \beta \in E$,
  \item $[\alpha, \beta] = 0$ for all $\alpha, \beta \in S$ with $\alpha \beta \notin E$, and 
  \item  $[u, v] = 0$ for all $u, v \in \lien_\mathcal{G}^K$ with either $u \in W$ or $v \in W$.  
\end{enumerate}

Recall that the automorphism group of a Lie algebra over the field $K$ is a $K$-linear algebraic group, i.e.\ it is given by the $K$-rational points of a subgroup $G$ of the general linear group over a complex vector space which is defined as the zero set of polynomial equations over $K$. We denote the subgroup of $K$-rational points as $G(K)$. For more details and terminology about these groups, we refer to \cite{bore91-1,hump81-1}. In the remainder of this paper we consider only the Zariski topology on linear algebraic groups, meaning that for example $\GL(V)$ is connected. 

If we denote by $$T = \{\varphi \in \Aut(\lien_\G^K) \mid \varphi(V) = V \}$$ and $$U = \{\varphi \in \Aut(\lien_\G^K) \mid \forall x \in \lien_\G^K: \hspace{1mm}  \varphi(x) - x \in W  \},$$ then both $T$ and $U$ are $K$-linear algebraic subgroups of $\Aut(\lien_\G^K)$. The automorphism group is equal to the semidirect product $\Aut(\lien_\G^K) = U \rtimes T$ by \cite[Proposition 2.1.]{dm05-1}. Note that $U$ consists of unipotent elements and hence forms a subgroup of the unipotent radical of $\Aut(\lien_\G^K)$. 

Under the natural projection map $p: T \to \GL(V) = \GL\left(\faktor{\lien_\G^K}{[\lien_\G^K,\lien_\G^K]}\right)$, we get that the image $G = p(T)$ is a $K$-linear algebraic group containing the diagonal matrices. In fact as shown in \cite{dm05-1}, one can see that this property characterizes the class of $2$-step nilpotent Lie algebras associated to graphs. Moreover, the map $p$ is injective and hence every element $g \in G$ uniquely corresponds to an automorphism $\varphi \in T$ with $p(\varphi) = g$. The groups $T$ and $G$ are hence isomorphic as $K$-linear algebraic groups.

\begin{example}
If $\G$ is the complete graph on $n$ vertices, then $\lien_\G^K$ is isomorphic to the free $2$-step nilpotent Lie algebra on $n$ generators. If $\G$ is the discrete graph on $n$ vertices, then $\lien_\G^K \approx K^n$ is an abelian Lie algebra. Hence Lie algebras associated to graphs intermediate between free nilpotent Lie algebras and abelian Lie algebras. Note that in both these cases the map $p: T \to \GL(V)$ is surjective.
\end{example}

In the next paragraphs, we study general linear algebraic groups containing the subgroup of diagonal matrices and afterwards apply these results to the specific case of Lie algebras associated to graphs.
  
\subsection{Linear algebraic groups containing $D$} \label{relations}

In this section, $G \le \GL(V)$ is a linear algebraic group defined over a field $K$ of characteristic $0$, where $V$ is a finite dimensional complex vector space. At the end of the section, we will interpret the results for the $K$-rational points of the group $G$, which corresponds to vector spaces over the field $K$.

We fix a basis $S$ for the vector space $V$ and denote by $D= D_S$ the subgroup of $\GL(V)$ consisting of all the {\em diagonal endomorphisms with respect to $S$}, i.e.  the linear endomorphisms  given by the  diagonal matrices with respect to the basis $S$. Then $D_S$ is a maximal torus of $\GL(V)$, i.e. a maximal connected abelian subgroup consisting of semisimple elements. The identity map on $V$ is denoted as $I_V$. The goal is to describe the groups $G$ with $D_S \le G$ or equivalently the linear algebraic groups over $K$ which contain a $K$-split torus of maximal rank. 

We start by recalling the work of \cite{dm05-1} about these connected linear algebraic groups, in particular the definition of the partial order relation $\prec$ associated to them. Using the relation $\prec$, we compute the normalizer in $\GL(V)$ of these groups in terms of permutation matrices, leading to a general description.

\paragraph{\textbf{The connected case}} For every $\alpha, \beta \in S$, we denote by $E_{\alpha\beta} \in \End(V)$ the linear map which is defined by $$E_{\alpha\beta}(\gamma) = 
\begin{cases} \hspace{1mm} \alpha & \text{if } \gamma = \beta, \\ \hspace{1mm}
    0           & \text{if } \gamma \neq \beta,
\end{cases}$$ for every basis vector $\gamma \in S$. Let $\lie$ denote the Lie algebra of $G$, which we consider as a Lie subalgebra of $\End(V)$.

First we  define a relation $\prec$ on $S$, using the maps $E_{\alpha\beta}$, as follows : $$ \text{ For } \alpha, \beta \in S, \text{ we say that }  \alpha \prec \beta  \text{ whenever } E_{\alpha \beta} \in \lie.$$ 
 It is easy to see that $\prec$ is a reflexive  (since  $D_S \le G$) and transitive relation on $S$. We start by giving an equivalent definition for the relation $\prec$ depending on the linear algebraic group $G$ itself.
\begin{proposition}
\label{1dimsubgroup}
Let $G \le \GL(V)$ be a linear algebraic group containing the subgroup $D_S$ of  all the diagonal endomorphisms with respect to the basis $S$ of $V$. Let $\alpha, \beta \in S$ with $\alpha \neq \beta$. The following statements are equivalent:
\begin{enumerate}
\item $\alpha \prec \beta$
\item $I_V + t E_{\alpha\beta} \in G(K)$ for every $t \in K$.
\end{enumerate}
\end{proposition}
 \begin{proof}
If $\alpha \prec \beta$ and $\alpha \neq \beta$, then by definition of $\prec$, we have that $E_{\alpha\beta} \in \lie$ and hence  $\exp(tE_{\alpha\beta}) \in G$ for $t \in K$. We know that $(E_{\alpha,\beta})^2 = 0$ as $\alpha \neq \beta$ which implies that  $I_V + t E_{\alpha\beta} = \exp(tE_{\alpha\beta}) \in G(K)$ for every $t \in K$. 

Conversely, assume that $I_V + t E_{\alpha\beta} \in G$ for every $t \in K$ and $\alpha \neq \beta$. Since $\Q \subset K$ and $\Q$ is dense in $\R$, we have $\exp(t E_{\alpha\beta}) \in G$ for all $t \in \R$ and hence $E_{\alpha\beta}  \in \lie$. 
\end{proof}

This shows that the relation $\prec$ does not depend on the field $K$ we are working with. The group of permutations on $S$ preserving $\prec$ is important for the remainder of the paper.
%

\begin{definition}

Let $A$ denote  a  finite set with a relation $R$ and let $\Perm(A)$ denote the set of all permutations of $A$.                   
A permutation $\sigma \in \Perm(A)$ {\em preserves the relation $R$}  if for $x, y \in A$,   $\sigma(x) R \sigma(y)$ whenever $xRy$.  By $\Perm(A, R)$ we denote the set of all permutations of $A$ preserving the relation $R$.
\end{definition}
\noindent The set $\Perm(A, R)$ forms a group under composition because the set $A$ is assumed to be finite.

We define now an equivalence relation $\sim$ on $S$ as follows. For $\alpha, \beta \in S$, we say that $\alpha \sim \beta$ if either $\alpha = \beta$ or both $E_{\alpha \beta}$ and $E_{\beta \alpha}$ are in $\lie$. Equivalently, $\alpha \sim \beta$ if and only if  $\alpha \prec \beta$ and $\beta \prec \alpha$. Let $\{S_{\lambda}\}_{\lambda \in \Lambda}$  denote the set of all equivalence classes under this equivalence relation and will be called {\em coherent components}.  With abuse of notation, we will denote a coherent component $S_\lambda$ by $\lambda$.  By construction the relation $\prec$ induces a partial order $\oprec$ on $\Lambda$ defined as follows:  $\lambda \hspace{1mm}\oprec \hspace{1mm} \mu$ for $\lambda, \mu \in \Lambda$ if  $\alpha \prec \beta$ for some (and hence every) $\alpha \in \lambda$ and $\beta \in \mu$. Moreover, the elements of $\Lambda$ can be enumerated as $\lambda_1, \ldots, \lambda_k$ such that  if $\lambda_i  \hspace{1mm} \oprec  \hspace{1mm} \lambda_j$ then  $i < j$. 
Any permutation $\sigma \in \Perm(S,\prec)$ induces a permutation $\bar{\sigma} \in \Perm(\Lambda, \oprec)$ on the equivalence classes. 

For $\lambda \in \Lambda$, we  denote  a subspace of $V$ spanned by the elements of the equivalence class $\lambda$ by $V_\lambda$. Note that $V = \displaystyle \bigoplus_{\lambda \in \Lambda} V_\lambda$. We embed the subgroup $\GL(V_\lambda)$ into $\GL(V)$ by taking the automorphism as the identity on each $V_\mu,  \lambda \neq \mu \in \Lambda$.

\begin{theorem}[\cite{dm05-1}]
	\label{conncomp}
Let $G \le \GL(V)$ be a connected linear algebraic group containing the subgroup $D$ of diagonal matrices, then
 \begin{equation}
 G =  \left( \displaystyle \prod_{\lambda \in \Lambda}\GL (V_{\lambda}) \right) M
\label{G^0}
\end{equation}
where  $M$ is the unipotent radical of $G$.
\end{theorem} The group $M$ is generated by the elements $I_V + t E_{\alpha,\beta}$ with $\alpha \prec \beta$ and $\beta \nprec \alpha$. We note that the subgroup $L = \displaystyle \prod_{i=1}^k\GL (V_{\lambda_i}) $ is a maximal reductive subgroup of $G$, which is also called a Levi subgroup of $G$. The original result in \cite{dm05-1} is for connected Lie groups, but it implies the above result for linear algebraic groups.

The following lemma is useful for studying the full automorphism group of Lie algebras associated to graphs. Let $V$ be any finite-dimensional vector space given as a direct sum $V = \displaystyle \bigoplus_{i \in I} V_i$, then we write $\hat{V}_j = \displaystyle \bigoplus_{\substack{i \in I \\ i \neq j}} V_i$. Again, we embed $\GL(V_j)$ as a subgroup of $\GL(V)$ in the natural way by taking identity on the other components, so the elements $d \in \GL(V_j)$ are exactly the ones for which $d(V_j)=V_j$ and $d$ is the identity on $\hat{V}_j$.
\begin{lemma}
	\label{otherdefpi}
	Let $V = \displaystyle \bigoplus_{i \in I} V_i$ be any finite-dimensional vector space, given as a direct sum of subspaces. For every $g \in \GL(V)$, the following are equivalent: $$g \GL(V_j) g^{-1} = \GL(V_k) \iff g(V_j) = V_k \text{ and } g(\hat{V}_j) = \hat{V}_k.$$
\end{lemma}
\noindent The proof is immediate, but we present it for completeness.
\begin{proof}
First assume that $g \GL(V_j) g^{-1} = \GL(V_k)$. Consider the map $d: V \to V$ given by $d(x) = x$ for $x \in \hat{V}_j$ and $d(y) = 2y$ for $y \in V_j$. Since $d \in \GL(V_j)$, we have $d^\prime = g d g^{-1} \in \GL(V_k)$ where $d^\prime$ has eigenvalues $1$ and $2$ just as the map $d$. In particular, by considering the eigenspaces for eigenvalue $1$ and $2$, we have $\hat{V}_k \subset g(\hat{V}_j)$ and $g(V_j) \subset V_k$. The reverse inclusions follow from considering $g^{-1}$, thus the first implication follows.

Now assume that $g(V_j) = V_k$ and $g(\hat{V}_j) = \hat{V}_k $. It suffices to show that the inclusion $g \GL(V_j) g^{-1} \subset \GL(V_k)$ holds, since we can apply the same statement to the element $g^{-1}$. If $d \in \GL(V_j)$, then $d(V_j)=V_j$ and so $g d g^{-1} \left(g(V_j)\right) = V_k$. Similarly since $d$ is the identity map on $\hat{V}_j$, then $g d g^{-1}$ is the identity map on $g(\hat{V}_j) = \hat{V}_k$ and thus the inclusion follows.
\end{proof}
We will apply this lemma on the decomposition $V = \displaystyle \bigoplus_{\lambda \in \Lambda} V_\lambda$ for finding the induced action on $\Lambda$.
\paragraph{\textbf{The general case}} 
For any permutation $\sigma$ of $S$, we denote by $P_\sigma \in \GL(V)$ the element defined by $$P_{\sigma}(\alpha) = \sigma(\alpha)$$ for all $\alpha \in S$. The matrix of $P_\sigma$ with respect to $S$ is the permutation matrix corresponding to the permutation $\sigma$. Note that at some places in literature, this is called the permutation matrix corresponding to the permutation $\sigma^{-1}$. Denote  by $$P: \Perm(S) \to \GL(V)$$ the group homomorphism given by $P(\sigma) = P_\sigma$. An easy computation shows that $  P_\sigma \comp E_{\alpha\beta} \comp P_{\sigma^{-1}} = E_{\sigma(\alpha)\sigma(\beta)}$ for all $\sigma \in \Perm(S)$ and for all $\alpha, \beta \in S$. In particular, for $ \sigma \in \Perm(S, \prec)$,  we have that \begin{align}\label{eq:induced}
\begin{split}
P_\sigma(V_\lambda) &= V_{\overline{\sigma}(\lambda)} \\P_\sigma(\hat{V}_\lambda) &= \hat{V}_{\overline{\sigma}(\lambda)}\end{split}
 \end{align} for every coherent component $\lambda$, which we will need later in the paper. Here we recall that $\overline\sigma \in \Perm(\Lambda, \oprec)$ is the permutation induced  by $\sigma$ and $\hat{V}_\lambda  = \displaystyle \bigoplus_{\substack{\mu \in \Lambda \\ \mu \neq \lambda}} V_\mu.$

\smallskip
The following fact about linear algebraic groups plays an important role.

\begin{proposition}         
\label{permute}
Let $G \leq GL(V)$ be a linear algebraic group which contains $D= D_S$ as a subgroup. If $G^0$ denotes the connected component of identity in $G$, then  \begin{equation}\label{S1}G =  G^0 \hspace{1mm} \big( G \cap P(\Perm(S)) \big)= \big( G \cap P(\Perm(S)) \big) \hspace{1mm}  G^0.\end{equation}
\end{proposition}

\begin{proof}The second equality of equation (\ref{S1}) follows from the fact that $G^0$ is a normal subgroup of $G$. We prove the first equality by showing two inclusions. It is immediate that $ G^0   \big( G \cap P(\Perm(S)) \big)\subset G$, so it suffices to show the reverse inclusion.

First, for the sake of completeness, we include a proof of a standard known fact that the normalizer of $D$ in $\GL(V)$ is given by \begin{align}\label{normalizerD}
N_{GL(V)}(D) &=  D P(\Perm(S)).\end{align} One inclusion is immediate, since for $\sigma \in \Perm(S)$ and $d \in D$, we have $P_\sigma d P_\sigma^{-1} \in D$. 
For the other inclusion, suppose that $g \in N_{GL(V)}(D)$ and $d \in D$ is given by $d (\alpha) = d_\alpha \alpha$ for all $\alpha \in S$ with the property $d_\alpha \neq d_\beta$ if $\alpha \neq \beta$, then we have  $g d g^{-1} = d' \in D$. Since both $d$ and $d'$ have the same set of eigenvalues, the set of diagonal entries of $d'$ is $\{ d_\alpha : \alpha \in S\}$. Denote by $\sigma \in \Perm(S)$ satisfying  $d'(\sigma(\alpha)) = d_\alpha \sigma(\alpha)$ for each $\alpha \in S$.  Now $d' g (\alpha) = g d(\alpha) = d_\alpha g(\alpha)$ and hence  $g(\alpha)$ is an eigenvector of  $d'$ corresponding to an eigenvalue $d_\alpha$ for each $\alpha \in S$. Since the $d_\alpha$'s are all distinct, we have $g(\alpha) = a_\alpha \sigma(\alpha)$ for some $a_\alpha \neq 0$ for each $\alpha \in S$. This proves that $g = h P_\sigma$ where $h(\sigma(\alpha)) = a_\alpha \sigma(\alpha)$ for each $\alpha \in S$, so $h \in D$ and hence proving equation (\ref{normalizerD}).

Now take any element $g \in G$. The subgroup $g D g^{-1}$ is a maximal torus in $G^0$ and thus there exists $h \in G^0$ with $h g D g^{-1} h^{-1} = D$ or equivalently with $hg \in N_{GL(V)}(D)$. Because of equation (\ref{normalizerD}), the element $hg$ is of the form $hg = d P_\sigma$ with $d \in D$ and $\sigma \in \Perm(S)$. We conclude that $g = h^{-1} d P_\sigma \in G^0  \big( G \cap P(\Perm(S)) \big)$ is of the desired form.
\end{proof}

By applying Proposition \ref{permute} to the normalizer of a connected linear algebraic group $G = G^0$, we get the following theorem.

\begin{theorem}
\label{normalizer}
Let $G$ be a connected linear algebraic group which contains the diagonal matrices $D= D_S$ as a subgroup. The normalizer of $G$ in $\GL(V)$ is equal to $$N_{\GL(V)}(G) = P(\Perm(S,\prec)) \hspace{1mm} G.$$
\end{theorem}

\begin{proof}
Since the normalizer of an algebraic group in $\GL(V)$ is itself an algebraic group by \cite[Page 59]{hump81-1}, Proposition \ref{permute} implies that $N_{\GL(V)}(G) = P(F) G$ where $P(F)$ is the subgroup of $P(\Perm(S))$ which normalizes $G$. It remains to check that $P(F) = P(\Perm(S,\prec))$. We first assume that $P_\sigma \in P(F)$ or thus $P_\sigma G P_{\sigma^{-1}} = G$. Note that for every $I_V + E_{\alpha\beta}$, it holds that $$P_\sigma (I_V + E_{\alpha\beta}) P_{\sigma^{-1}} = I_V + E_{\sigma(\alpha)\sigma(\beta)}$$ and thus that $I_V + E_{\alpha \beta} \in G$ if and only if $I_V + E_{\sigma(\alpha)\sigma(\beta)} \in G$. From Proposition \ref{1dimsubgroup} it then follows that $\alpha \prec \beta$ if and only if $\sigma(\alpha) \prec \sigma(\beta)$ or thus that $\sigma \in \Perm(S,\prec)$. 

Conversely,  asume that $\sigma \in \Perm(S,\prec)$, then we will show that $P_\sigma G P_{\sigma^{-1}} = G$ from the explicit form for $G$ in Theorem \ref{conncomp}. First consider the generators of $M$ which are given by $I_V + t E_{\alpha\beta}$ with $\alpha \prec \beta, \beta \nprec \alpha$ and $t \in \C$. In this case $$P_\sigma (I_V + t E_{\alpha\beta}) P_{\sigma^{-1}} =  I_V + t E_{\sigma(\alpha)\sigma(\beta)} \in M,$$ which shows that $P_\sigma M P_{\sigma^{-1}} = M$. On the other hand, for the subgroup $\GL(V_\lambda)$ with $\lambda \in \Lambda$, we get that $P_\sigma \GL(V_\lambda) P_{\sigma^{-1}} = \GL(V_{\bar{\sigma}(\lambda)})$ by Lemma \ref{otherdefpi} and equation (\ref{eq:induced}), where $\overline\sigma \in \Perm(\Lambda, \oprec)$ is the permutation induced  by $\sigma$. Hence the conclusion follows.\end{proof}

In this way we get a description of all linear algebraic groups containing $D$.

\begin{corollary}
	\label{defofpi}
	Let $G$ be a linear algebraic group containing $D=D_S$ and with $G^0$ as the Zariski connected component of the identity. Then $G$ is given by  $$G= P(F) G^0$$ where $F \le \Perm(S,\prec)$ is the finite subgroup given by $F = \{\sigma \in \Perm(S,\prec) \mid P_\sigma \in G\}$.
\end{corollary}

\begin{proof}
	This is an immediate consequence of Theorem \ref{normalizer}. Indeed for if $g \in G$, then $g = p_{\sigma} h$ for some $\sigma \in  \Perm(S,\prec)$ and $h \in G^0$.
\end{proof}

 \begin{remark}
 Theorem \ref{G^0} and Corollary \ref{defofpi} allow us to order the basis $S$, by ordering the coherent classes as $\lambda_1, \ldots, \lambda_k$,  such that the matrix of an element $g \in G$ with respect to $S$ is of the form
\[ \text{\large $P$}
\begin{pmatrix}
  A_{\lambda_1}&\rvline &  A_{12}  &\rvline & A_{13}  &\rvline & \cdots & \rvline &A_{1k}  \\
  
\hline

 \text{\Large0} &\rvline& A_{\lambda_2} & \rvline & A_{23} &\rvline & \cdots & \rvline &A_{2k}\\

\hline

 \text{\Large0} &\rvline&  \text{\Large0}& \rvline &  A_{\lambda_3} &\rvline & \cdots & \rvline &A_{3k}\\

\hline

\vdots &\rvline&\vdots & \rvline &  \vdots &\rvline & \ddots & \rvline &\vdots\\

\hline

 \text{\Large0} &\rvline&  \text{\Large0}& \rvline &   \cdots &\rvline &   \text{\Large0}& \rvline &A_{\lambda_k}\\

   \end{pmatrix},\]

where $\text{\large $P$}$ is a permutation matrix, $A_{\lambda_i} \in \GL(V_{\lambda_i})$ and $A_{ij} = \text{\Large0}$, a zero matrix,  if $\lambda_i \hspace{1mm} \overline{\nprec} \hspace{1mm} \lambda_j$.
\end{remark}

In order to fully understand the group $G$, we need a description of the subgroup $P(F) \cap G^0$. Let $\overline{\Perm(S, \prec)}$ denote the subgroup of $\Perm(\Lambda, \oprec)$ consisting of the permutations $\overline{\sigma}$ induced by $\sigma \in \Perm(S, \prec)$. The intersection of $P \left(\Perm(S,\prec) \right)$ and $G^0$ is described by the following lemma.

\begin{lemma}
	For every $\sigma \in \Perm(S, \prec)$ it holds that $\overline{\sigma} = 1$ if and only if $P_\sigma \in G^0$. 
\end{lemma}

\begin{proof}
First assume that $\sigma \in \Perm(S,\prec)$ such that $\overline{\sigma} = 1$. Then for every $\lambda \in \Lambda$,  we have   $P_\sigma(V_\lambda) =  V_{\overline{\sigma}(\lambda)} = V_\lambda$. In particular, we have that $P_\sigma \in \displaystyle \prod_{i=1}^k \GL(V_{\lambda_i}) \subset G^0$. 

For the other implication, assume that $P_\sigma \in G^0$. Note that the transpose $P_\sigma^T = P_{\sigma^{-1}}$ and therefore $P_\sigma^T\in G^0$. By using Theorem \ref{conncomp}, we get that $P_\sigma \in \displaystyle \prod_{i=1}^k \GL(V_{\lambda_i})$, which implies that $\overline{\sigma} = 1$.
%
\end{proof}

In other words this lemma states that $$ G^0 \cap P(\Perm(S,\prec)) = P \big(\{ \sigma \in \Perm(S,\prec) \mid \overline{\sigma} = 1 \big).$$ The following is hence an immediate consequence.

\begin{corollary}
	\label{cor:action}
Every linear algebraic group $G$ containing the group of diagonal matrices $D= D_S$ satisfies $$\faktor{G}{G^0} \approx \bar{F}$$ with $\bar{F}$ a finite subgroup of $\overline{\Perm(S, \prec)} \le \Perm(\Lambda, \oprec)$.
\end{corollary}
\begin{proof}
The group $G$ is a subgroup of $P(\Perm(S,\prec)) G^0 = N_{\GL(V)}(G^0)$, so every $g \in G$ is of the form $P_\sigma h$ with $\sigma \in \Perm(S,\prec)$ and $h \in G^0$. Define the map $\pi: G \to \overline{\Perm(S, \prec)}$ which maps the element $g = P_\sigma h$ to $\overline{\sigma}$. This map is well-defined, since if $P_\sigma h = P_\tau h'$ with $\sigma, \tau \in \Perm(S, \prec)$ and $h,h' \in G^0$, then $P_{\sigma^{-1} \tau} = P_{\sigma^{-1}} P_\tau = h h'^{-1} \in G^0$ and thus $\overline{\sigma} = \overline{\tau}$ by the previous lemma. One can check that $\pi$ is a group homomorphism as  $G^0$ is normal in $G$. Therefore we have an induced map $\overline{\pi}: \faktor{G}{G^0} \to \overline{\Perm(S, \prec)}$ which is an injective group homorphism and the corollary now follows immediately.
\end{proof}
Corollary \ref{cor:action} shows that there is a natural action of $G$ on the equivalence classes $\Lambda$. We note that in general $\overline{\Perm(S, \prec)} \le \Perm(\Lambda, \oprec)$ is a strict subgroup as we will show in Example \ref{QuotientGraph2}. 
We denote by $\pi: G \to \Perm(\Lambda, \oprec)$ the natural map defined by the previous corollary. It plays an important role in the next sections for studying the actions of finite groups. Since $\overline{\sigma} = \pi(P_\sigma)$ for every $\sigma \in \Perm(S,\prec)$, we have that $P_\sigma(V_\lambda) = V_{\overline{\sigma}(\lambda)} = V_{\pi(P_\sigma)(\lambda)}$. Lemma \ref{otherdefpi} implies that $$P_\sigma \GL(V_\lambda)  P_{\sigma}^{-1} = \GL(V_{\overline{\sigma}(\lambda)}) =  \GL(V_{\pi(P_\sigma)(\lambda)}) $$ for every $P_\sigma \in G$. This alternative way of looking at the map $\pi$ is how we will use it in the proof of the main result.

Note that the previous results were for linear algebraic groups $G \le \GL(V)$ over the complex numbers defined over $K$, but the main application is the automorphism group $\Aut(\lien_\G^K)$ of a Lie algebra over a subfield $K \subseteq \C$. Hence we are often interested in the subgroup of $K$-rational points $G(K)$ for $K \subseteq \C$. Since $P(F) \subseteq G(\Q)$, it follows immediately that $$G(K) = G^0(K) P(F) = \left( \displaystyle \prod_{i=1}^k\GL (V_{\lambda_i},K) \right) M(K) P(F).$$

\subsection{Full automorphism group of Lie algebras associated to graphs}

 We now apply the previous results to the  automorphism group of a Lie algebra associated to a graph as introduced in \cite{dm05-1}. We refer to Subsection \ref{subsec:graphs} for notations. Let $\G = (S,E)$ be a graph and $\lien_\G^K$ the corresponding Lie algebra over a field $K$ of characteristic $0$. Let  $\Aut(\G)$ denote the group of all graph automorphisms of $\G$. 
 
 We recall two subgroups of  $\Aut(\lien_\G^K)$ introduced as before:
 $$T = \{\varphi \in \Aut(\lien_\G^K) \mid \varphi(V) = V \} \text{ and } U = \{\varphi \in \Aut(\lien_\G^K) \mid \forall x \in \lien_\G^K: \hspace{1mm}  \varphi(x) - x \in W  \},$$ 
 where $W = \text{Span}\{[\alpha, \beta] \mid \alpha \beta \in E\}$. We will apply Corollary \ref{defofpi} to $G = p(T)$ where $p$ is the natural projection map $p: T \to \GL(V) = \GL\left(\faktor{\lien_\G^K}{[\lien_\G^K,\lien_\G^K]}\right)$, but first we describe the elements $P_\sigma \in G$.

\begin{lemma}
For every $\sigma \in \Perm(S)$, the following are equivalent: $$\sigma \in \Aut(\G) \iff P_\sigma \in G = p(T).$$
\end{lemma}
\begin{proof}
First assume that $\sigma \in \Aut(\G)$ then $P_\sigma: V \to V$ extends to an automorphism of $\lien_{\G}^K$. In particular, $P_\sigma \in p(T) = G$. For the other implication, note that $P_\sigma = p(\varphi)$ for $\varphi \in T \subset \Aut(\lien_\G^K)$, then $$\varphi([\alpha,\beta]) = [P_\sigma(\alpha),P_\sigma(\beta)] = [\sigma(\alpha), \sigma(\beta)].$$ Since we have the equivalence $\alpha\beta \in E \iff [\alpha,\beta] \neq 0$ and $\varphi$ maps non-zero elements on non-zero elements, this implies that $\sigma \in \Aut(\G)$.
\end{proof}

Combining Corrolary \ref{defofpi} with the above lemma, we get the following description of the automorphism group of $\lien_\G^K$.

\begin{corollary}
Let $K$ be any field of characteristic $0$, then the automorphism group $\Aut(\lien_\G^K)$ is equal to $$ \Aut(\lien_\G^K) = U T$$ with $$T \approx P\left(\Aut(\G)\right) \hspace{1mm} G^0(K) = P(\Aut(\G)) \displaystyle \prod_{i=1}^k\GL(V_{\lambda_i},K)M(K),$$ where the isomorphism is given by  the natural projection $p$.

\end{corollary}
\noindent So the automorphism group of $\lien_\G^K$ over any field $K$ of characteristic $0$ is completely described in terms of the graph $\G$. Note that the Levi subgroup $L$ of $G$ is also a Levi subgroup of $\Aut(\lien_\G^K)$ because $U$ is a unipotent normal subgroup.

\smallskip

Consider the morphism $\pi: G  \to \Perm(\Lambda, \oprec)$ as introduced under Corollary \ref{defofpi}. Lemma \ref{otherdefpi} shows that $g (V_\lambda) = V_{\pi(g)(\lambda)}$ for every $g \in P(\Perm(S,\prec))$. More general, this relation holds for every element in the Levi subgroup $L = \displaystyle \prod_{i=1}^k\GL(K,V_{\lambda_i}) P(\Aut(\G))$ by Lemma \ref{otherdefpi}, leading to an action $L \curvearrowright \Lambda$ which is central in the proof of the main result. In order to work with this action, we recall how we can find the relation $\prec$ directly from the graph $\G$.

\noindent \textbf{Coherent components for graphs.} Let $\G = (S, E)$ denote a finite simple graph. We will now recall some observations from \cite{dm05-1} interpreting the relation $\prec$ on the set of vertices $S$. We will also introduce some  terminology which would be useful in our applications later in the paper. 
 For $\alpha \in S$, we define the {\em open neighborhood}  and  {\em closed neighborhood}   of $\alpha$, respectively, as follows:
\[ \Omega_{\G}'(\alpha) = \{ \beta \in S \,\,| \,\,\alpha \beta \in E\}\,\,\,\text{ and } \,\,\Omega_{\G}(\alpha) =  \Omega_{\G}'(\alpha) \cup \{\alpha\}.\]

For $\alpha, \beta \in S$, we have $\alpha \prec \beta$ if and only if $\Omega_{\G}'(\alpha) \subseteq \Omega_{\G}(\beta)$ by \cite[Proposition 4.1]{dm05-1}. Consequently, $\alpha \sim \beta$ if and only if  $\Omega_{\G}'(\alpha) \subseteq \Omega_{\G}(\beta)$ and $\Omega_{\G}'(\beta) \subseteq \Omega_{\G}(\alpha)$.  As before, we denote the set of coherent components (equivalence classes) by $\Lambda$. It can be checked that the induced subgraph on each coherent component $\lambda \in \Lambda$ is either complete or discrete. More precisely, for every $\lambda \in \Lambda$, either $\alpha \beta \in E$ for all distinct $\alpha, \beta \in \lambda$ or  $\alpha \beta \notin E$ for any  $\alpha, \beta \in \lambda$.   Moreover, for $\lambda \neq \mu \in \Lambda$, if there exist $\alpha \in \lambda, \beta \in \mu$ with $\alpha \beta \in E$, then   we have  $\gamma\delta \in E$ for all $\gamma \in \lambda, \delta \in \mu$.  These properties lead to a notion of {\em quotient graph} $\overline{\G}$ whose vertex set is $\Lambda$ and the edge set $\mathcal{E}$ is given by \[\mathcal{E} = \{ \lambda \mu \mid  \text{ there exist } \alpha \in \lambda, \beta \in \mu  \text{ with } \alpha \beta \in E \}.\]
Note that $\overline{\G}$ might have vertices with loops. Indeed, if a subgraph of $\G$ induced on $\lambda$ is complete and $\lambda$ contains at least two vertices, then $\lambda \lambda \in \mathcal{E}$.

The induced partial order $\oprec$ on  $\Lambda$  is given by 
$$\lambda \hspace{1mm} \oprec\hspace{1mm} \mu\,\,\,\, \iff\,\,\,\,    \Omega'_{\G}(\alpha) \subseteq \Omega_{\G}(\beta) \text{ for some } \alpha \in \lambda, \beta \in \mu. $$ 
We arrange the  coherent components  $\lambda'$s as $\lambda_1, \ldots, \lambda_k$ such that  if $\lambda_i \hspace{1mm} \oprec \hspace{1mm} \lambda_j$, then $i \leq j$. 

In general, we have the inclusion of groups $\Aut(\G) \subset \Perm(S, \prec)$ and $\overline{\Perm(S,\prec)} \subset \Perm(\Lambda, \oprec)$. We give some examples showing that these can be equalities or strict inclusions depending on the graph $\G$.

\begin{example}\label{QuotientGraph1}

Consider the graph $\G$ with vertices $\{\alpha_1, \alpha_2, \beta_1, \beta_2, \beta_3, \gamma_1, \gamma_2 \}$ as follows.

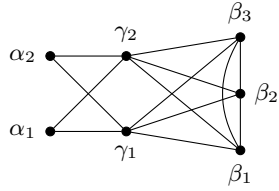
\[
\begin{tikzpicture}[font=\small,baseline=-4]
\node[vertex,label={left:$\alpha_1$}] (1) at (0,0) {};
\node[vertex,label={left:$\alpha_2$}] (2) at (0,1) {};

\node[vertex,label={below:$\gamma_1$}] (3) at (1, 0) {};
\node[vertex,label={above:$\gamma_2$}] (4) at (1,1) {};

\node[vertex,label={below:$\beta_1$}] (5) at (2.5,-.25) {};
\node[vertex,label={right:$\beta_2$}] (6) at (2.5,0.5) {};
\node[vertex,label={above:$\beta_3$}] (7) at (2.5,1.25) {};

\path[-]
 (1) edge (3)
 (1) edge (4)

 (2) edge (3)
 (2) edge (4)
  
 (3) edge (5)
 (3) edge (6)
 (3) edge (7)
 
 (4) edge (5)
 (4) edge (6)
(4) edge (7)

 (5) edge (6)
(5) edge [bend left]   node[left] {}  (7)

 (6) edge (7);
\end{tikzpicture}
\]

Looking at the neighborhoods of the vertices, we get 3 coherent components, say $\lambda, \nu$ and $\mu$, where $\lambda = \{\alpha_1, \alpha_2\}, \nu= \{\gamma_1, \gamma_2\}$ and $\mu = \{ \beta_1, \beta_2, \beta_3\}$.
The corresponding quotient graph $\overline{\G}$ is the graph
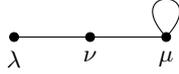
\[ 
\begin{tikzpicture}[font=\small,baseline=-4]
\node[vertex,label={below:$\lambda$}] (1) at (0,0) {};

\node[vertex,label={below:$\nu$}] (3) at (1, 0) {};

\node[vertex,label={below:$\mu$}] (5) at (2,0) {};
\path[-]
 (1) edge (3)
 (3) edge (5)
(5) edge [loopup] (5)
;
\end{tikzpicture}
\]

Since $\Omega'(\alpha_1)  \subset \Omega(\beta_1)$, we can see that  $\lambda \, \oprec \, \mu$.   The other pairs of the coherent classes are not comparable with respect to the relation $\oprec$. Hence we can arrange the coherent classes as $\lambda_1 = \lambda, \lambda_2 = \mu$ and $\lambda_3 = \nu$ so that $i \leq j$ whenever $\lambda_i \hspace{1mm} \oprec \hspace{1mm} \lambda_j$.  We could also arrange them in the order $\nu,\lambda,  \mu$ or $\lambda, \nu, \mu$.

In this case,    $\lambda \hspace{1mm} \oprec \hspace{1mm} \mu$ and no other pairs of coherent classes are comparable, hence the group $\Perm(\Lambda, \oprec)$ is trivial. The group $\Perm(S, \prec)$ consists of permutations $\sigma$ of $S$ satisfying $\sigma(\lambda_i) = \lambda_i$ for all $i \in \{ 1, 2, 3\}$. Thus $\Perm(S, \prec) = \Aut(\G)$ in this case.

\end{example}

\begin{example}\label{QuotientGraph2}
	Consider the graph $\G$ with a vertex set the same as in the Example \ref{QuotientGraph1} and the edge set as drawn below.
	
	
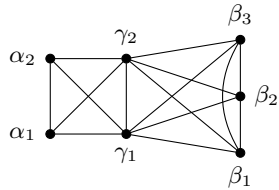
\[
\begin{tikzpicture}[font=\small,baseline=-4]
\node[vertex,label={left:$\alpha_1$}] (1) at (0,0) {};
\node[vertex,label={left:$\alpha_2$}] (2) at (0,1) {};

\node[vertex,label={below:$\gamma_1$}] (3) at (1, 0) {};
\node[vertex,label={above:$\gamma_2$}] (4) at (1,1) {};

\node[vertex,label={below:$\beta_1$}] (5) at (2.5,-.25) {};
\node[vertex,label={right:$\beta_2$}] (6) at (2.5,.5) {};
\node[vertex,label={above:$\beta_3$}] (7) at (2.5,1.25) {};

\path[-]
(1) edge (2)
 (1) edge (3)
 (1) edge (4)
 (2) edge (3)
 (2) edge (4)
  (3) edge (4)
   (3) edge (5)
 (3) edge (6)
 (3) edge (7)
 (4) edge (5)
 (4) edge (6)
 (4) edge (7)
 (5) edge (6)
 (6) edge (7)
 (5) edge [bend left]   node[left] {}  (7)

 ;

 \end{tikzpicture}
\]

There are three coherent classes in this graph, say $\lambda_1 = \{\alpha_1, \alpha_2\}, \lambda_2 = \{ \beta_1, \beta_2, \beta_3\}$ and $\lambda_3 = \{\gamma_1, \gamma_2\}$, and the quotient graph is as drawn below.

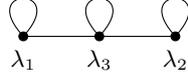
\[ 
\begin{tikzpicture}[font=\small,baseline=-4]
\node[vertex,label={below:$\lambda_1$}] (1) at (0,0) {};

\node[vertex,label={below:$\lambda_3$}] (3) at (1, 0) {};

\node[vertex,label={below:$\lambda_2$}] (5) at (2,0) {};
\path[-]
(1) edge [loopup] (1)

 (1) edge (3)
 
(3) edge  [loopup] (3)
 (3) edge (5)
(5) edge [loopup] (5)
;
\end{tikzpicture}
\]

The coherent classes satisfy $\lambda_1\, \oprec \,\lambda_3, \lambda_2\, \oprec \,\lambda_3.$
We note that $\lambda_1$ and $\lambda_2$ are not comparable.  Hence if $\lambda_i \hspace{1mm} \oprec \hspace{1mm} \lambda_j$, then $i \leq j$.

Now we will show that $\overline{\Perm(S, \prec)}$ is a proper subgroup of $\Perm(\Lambda, \oprec)$ in this case.  Consider a $2$-cycle $\tau = (\lambda_1 \,\, \lambda_2) \in \Perm(\Lambda)$. Then $\tau(\lambda_1) = \lambda_2 \, \oprec \, \lambda_3 = \tau(\lambda_3)$ and $\tau(\lambda_2) = \lambda_1 \, \oprec \, \lambda_3 = \tau(\lambda_3)$. This shows that $\tau \in \Perm(\Lambda, \oprec)$.  Now for every element  $\sigma \in \Perm(S)$, we have $\sigma(\lambda_1) \neq \lambda_2$ as $\# \lambda_1 \neq \# \lambda_2$. Hence $\tau$ is not induced by any element of $\Perm(S, \prec)$.   In particular it follows that  not every element of $\Perm(\Lambda,\oprec)$ is induced by a graph automorphism. In fact, as in Example \ref{QuotientGraph1}, we can show that $\Perm(S, \lambda) = \Aut(\G)$ by observing that every permutation in $\Perm(S, \prec)$ stabilizes each coherent component.
\end{example}

%

\begin{example}\label{QuotientGraph3}
	Consider a graph $\G$  as drawn below.
	
	
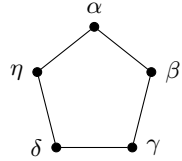
\[
\begin{tikzpicture}[font=\small,baseline=-4]
\node[vertex,label={above:$\alpha$}] (1) at (0,.85) {};
\node[vertex,label={right:$\beta$}] (2) at (.75,0.25) {};
\node[vertex,label={left:$\eta$}] (5) at (-.75, 0.25) {};
\node[vertex,label={right:$\gamma$}] (3) at (0.5, -.75) {};
\node[vertex,label={left:$\delta$}] (4) at (-.5, -.75) {};

\path[-]
(1) edge (2)
 (2) edge (3)
 (3) edge (4)
 (4) edge (5)
 (5) edge (1)
   ;
\end{tikzpicture}
\]

There are five coherent classes and no two vertices are comparable under the relation $\prec$. Therefore $\Perm(S, \prec) = \Perm(S)$ in this case. In particular, $\Aut(\G)$ is a proper subgroup of $\Perm(S, \prec)$. 
  \end{example}

%

%
%

\section{Anosov automorphisms commuting with finite subgroups}
\label{sec:mainresult}

In this part we focus on the Lie groups corresponding to the Lie algebras of the previous section. For $K = \mathbb  R$,  let $N_\mathcal{G}$ be the connected and simply connected nilpotent Lie group with $ \lien_\mathcal{G}^{\R}$ as its Lie algebra. We call {\em $N_\mathcal{G}$  a 2-step nilpotent Lie group associated to the graph $\mathcal{G}$}. The rational Lie algebra $\lien_\G^\Q$ corresponds to a lattice $\Gamma \le N_\G$, which is uniquely defined up to commensurability and the nilmanifold $\faktor{N_\G}{\Gamma}$ is said to be associated to the graph $\G$. In the paper \cite{dm05-1} the authors give an explicit form for such a lattice $\Gamma$, but this is not necessary for our purposes. Note that there are in general other nilmanifolds modeled on a Lie group associated to a graph, see Section \ref{sec:openQ}. In this section, we study which infra-nilmanifolds covered by a nilmanifold associated to a graph admit an Anosov diffeormorphism.

\begin{definition}
	\label{def-associated}
Let $\G$ be a graph and $N_\G$ the Lie group associated to $\G$. We say that an infra-nilmanifold $N_\G / \Gamma$ is \emph{associated to the graph $\G$} if the Lie algebra corresponding to the radicable hull $M^\Q$ of $M = N \cap \Gamma$ is isomorphic to $\lien_\G^\Q$. 
\end{definition}
\noindent By \cite{main15-1} the graph $\G$ is uniquely determined by the rational Lie algebra $\lien^\Q_\G$.

Theorem \ref{charkk} states that for studying Anosov diffeomorphisms on infra-nilmanifolds associated to graphs, we have to study Anosov automorphisms commuting with finite subgroups $H$ of the automorphism group $\Aut(\lien_\G^\Q)$. We start by constructing induced representations on the coherent components before getting to the proof of the main result.

\paragraph{\textbf{Induced representations on coherent components}} 
 Let $$L = P(\Aut(\G)) \displaystyle \prod_{i=1}^k \GL(V_{\lambda_i},\Q)$$ be the maximal reductive subgroup of $G(\Q)$ as above. We know that there is an algebraic isomorphism between $T$ and $G$, so we can also consider $L \le T \le \Aut(\lien_\G^\Q)$ as a Levi subgroup of $\Aut(\lien_\G^\Q)$.

 If $H$ is a finite subgroup of $\Aut(\lien_\G^\Q)$, then it is reductive and thus lies in a maximal reductive subgroup $L^\prime$. Since all maximal reductive subgroups of a linear algebraic group are conjugate in characteristic $0$, there exists $g \in G(\Q)$ such that $g L^\prime g^{-1} = L$ or thus $g H g^{-1} \subseteq L$. Without loss of generality, we can assume that $H$ is a finite subgroup of $L$. The same argument shows that if there exists an Anosov automorphism on $\lien_\G^\Q$ commuting with $H$, we can assume it lies in $L$ as well. Indeed, if $\varphi: \lien^\Q_\G \to \lien^\Q_\G$ is an Anosov automorphism, then also its semisimple part $\varphi_s$ is Anosov and commutes with $H$. Now by considering the reductive subgroup generated by $H$ and $\varphi_s$, we can assume that both lie in $L$. 
  
 Even stronger, if we consider the natural projection map $\psi: \Aut(\n_\G^\Q) \to L$ by taking the quotient by the unipotent radical, then we get that $\psi(H)$ is a finite subgroup of $L$. When we consider automorphisms as matrices in the standard basis on $\n_\G^\Q$, the map $\psi$ is given by taking the block diagonal part of a matrix. If $\varphi \in \Aut(\n_\G^\Q)$ is an Anosov automorphism, then $\psi(\varphi)$ will again be an Anosov automorphism lying in $L$ with the same eigenvalues. Moreover, if $H$ and $\varphi$ commute, also $\psi(H)$ and $\psi(\varphi)$ commute. So without loss of generality, we can take the projection $\psi$ in order to assume that both $H$ and $\varphi$ are elements of the Levi subgroup $L$. Even if we assume later that $H$ is a subgroup of $L$, we can state our main theorem for general representation $\rho: H \to \Aut(\lien_\G^\Q)$, where we first apply the map $\psi$ and get an Anosov automorphism for the original representation as explained above.

So from now on we assume that $H$ is a subgroup in $L$. By restricting the action at the end of the previous section, we get a map $\pi: H \to \Perm(\Lambda, \oprec)$ such that for every $h \in H$, we have  $$h (V_\lambda) = V_{\pi(h)(\lambda)}.$$ This gives an action of $H$ on the coherent components $\Lambda$, which we denote by $h\cdot \lambda = \pi(h)(\lambda)$. For this action of $H$ on $\Lambda$, we denote the  orbits by $\kappa_i = H \cdot \lambda_i$ with $1 \leq i \leq k$. 
Write $c(\lambda_i) = 1$ if the subgraph of the quotient graph $\overline{\G}$ (see Section \ref{sec:fullauto}) induced on the orbit $\kappa_i$ is an edgeless graph, otherwise $c(\lambda_i) = 2$. In other words, we look at the union of all coherent components in the orbit $\kappa_i$ which would give us a subset of the vertex set $S$ of $\G$. If there are no two vertices in that union which are adjacent in $\G$, then we write $c(\lambda_i) = 1$, otherwise, we write $c(\lambda_i) =2$. In conclusion, $c(\lambda_i) =2$ if and only if there is an edge (possibly a loop)  in the subgraph of $\overline{\G}$ induced by the orbit $\kappa_i$. 
Note that $c(\lambda_i)$ is equal to the nilpotency class of the subalgebra generated by the elements in the coherent components of the orbit $\kappa_i$.


Every orbit $\kappa_i = H \cdot \lambda_i$ also defines a subgroup $H_i \leq H$ given by the stabilizer of $\lambda_i$, i.e. $$H_i = \{h \in H \mid h(V_{\lambda_i}) = V_{\lambda_i}\}.$$ So the action of $H$ on $\Lambda$ determines $k$ representations $\rho_i: H_i \to \GL(V_{\lambda_i})$. These representations depend on the choice of the elements $\lambda_i \in \kappa_i$, but this does not influence our results below.


\begin{example}\label{clambda}
Consider a graph $\G$ with its quotient graph $\overline{\G}$ drawn as below.

	
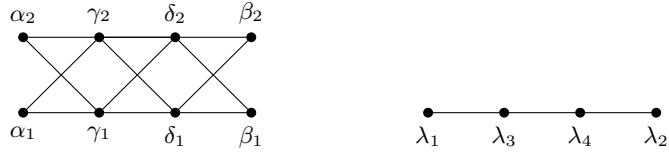
\[
\begin{tikzpicture}[font=\small,baseline=-4]
\node[vertex,label={below:$\alpha_1$}] (1) at (0,0) {};
\node[vertex,label={above:$\alpha_2$}] (2) at (0,1) {};

\node[vertex,label={below:$\gamma_1$}] (3) at (1,0) {};
\node[vertex,label={above:$\gamma_2$}] (4) at (1,1) {};

\node[vertex,label={below:$\delta_1$}] (5) at (2, 0) {};
\node[vertex,label={above:$\delta_2$}] (6) at (2,1) {};

\node[vertex,label={below:$\beta_1$}] (7) at (3, 0) {};
\node[vertex,label={above:$\beta_2$}] (8) at (3,1) {};

\path[-]
 (1) edge (3)
 (1) edge (4)
 (2) edge (3)
 (2) edge (4)
  (3) edge (5)
   (3) edge (6)
 (4) edge (5)
 (4) edge (6)
 (4) edge (6)
  (5) edge (7)
 (5) edge (8)
 (6) edge (7)
 (6) edge (8)
 ;
\end{tikzpicture}
\hspace{5em}
\begin{tikzpicture}[font=\small,baseline=-4]
\node[vertex,label={below:$\lambda_1$}] (1) at (0,0) {};
\node[vertex,label={below:$\lambda_3$}] (2) at (1,0) {};
\node[vertex,label={below:$\lambda_4$}] (3) at (2,0) {};
\node[vertex,label={below:$\lambda_2$}] (4) at (3,0) {};

\path[-]
 (1) edge (2)
 (2) edge (3)
 (3) edge (4)
  ;
\end{tikzpicture}
\]

Here the coherent classes are $\lambda_1 = \{ \alpha_1, \alpha_2\}, \lambda_2 = \{ \beta_1, \beta_2\}, \lambda_3 = \{ \gamma_1, \gamma_2\}$ and $\lambda_4 = \{ \delta_1, \delta_2\}$.
Suppose $H$ is a subgroup of $L$ generated by $P_\sigma$ where $\displaystyle{\sigma = \prod_{i=1}^2 (\alpha_i \, \beta_i) (\gamma_i \, \delta_i)}$, which is hence of order $2$. Then the action of $H$ on $\Lambda$ is given by \[ P_\sigma \cdot \lambda_1 = \lambda_2, \,\,P_\sigma \cdot \lambda_2 = \lambda_1,\, \,\,P_\sigma \cdot \lambda_3 = \lambda_4, \,\,P_\sigma \cdot \lambda_4 = \lambda_3.\] Hence $\kappa_1 =\{\lambda_1, \lambda_2\} = \kappa_2$ and $\kappa_3 =\{\lambda_3, \lambda_4\} = \kappa_4$. Consequently $c(\lambda_1) = 1 = c(\lambda_2)$ and $c(\lambda_3) = 2 = c(\lambda_4)$.  The stabilizer $H_i$ is trivial for each $i$ in this example.

\end{example}

 \begin{remark} \label{eigenvalues}
 	Although $L$ is isomorphic to a subgroup of $\Aut(\lien^\Q_\G)$, this isomorphism does not preserve the spectrum of an element. Take any $g \in G \le \GL(V)$ which uniquely corresponds to an automorphism $\varphi \in T \le \Aut(\lien_\G)$. Write $n_\lambda$ for the dimension of the subspace $V_\lambda$ with $\lambda \in \Lambda$. If $g \in L \subseteq G$ and $\mu_{\lambda,1}, \ldots, \mu_{\lambda, n_\lambda}$ are the eigenvalues of $g$ on $V_\lambda$, then the eigenvalues of $\varphi$ are equal to $\mu_{\lambda,i}$ for $\lambda \in \Lambda$ or $\mu_{\lambda,i} \mu_{\lambda^\prime,j}$ for $\lambda, \lambda^\prime \in \Lambda$. The latter eigenvalues only occur if either $\lambda \neq \lambda^\prime$ and these coherent components are connected by an edge in the quotient graph or if $\lambda = \lambda^\prime, i \neq j$ and the coherent component $\lambda$ is a complete graph. 
\end{remark}

\paragraph{\textbf{Main result}}

Recall that a matrix $A \in \GL(n,\Q)$ with eigenvalues $\mu_1, \ldots, \mu_n$ is called $c$-hyperbolic if for all $1 \leq l \leq c$ and $1 \leq i_j \leq n$  the product $$ \displaystyle \prod_{j=1}^l \vert \mu_{i_j} \vert \neq 1.$$ So $1$-hyperbolic corresponds to the classical notion of hyperbolic matrices. A linear automorphism with a $c$-hyperbolic matrix is called {\em $c$-hyperbolic automorphism.} The action of $H$ on the coherent components, the induced representations $\rho_i$ together with $c(\lambda_i)$ contain all the information about the existence of Anosov diffeomorphisms.

\begin{theorem}
Let $\G$ be a graph and $\lien^\Q_\G$ be the rational Lie algebra associated to this graph. Let $\rho: H \to L \le \Aut(\lien^\Q_\G)$ be a representation of a finite group and $\rho_i: H_i \to \GL(V_{\lambda_i})$ be the representations as introduced above. The following statements are equivalent.

\begin{center}
There exists an Anosov automorphism $\varphi \in \Aut(\lien^\Q_\G)$ commuting with every element of $\rho(H)$. \\
	$\Updownarrow$ \\
For every $1 \leq i \leq k$, there exists a $c(\lambda_i)$-hyperbolic integer-like automorphism $\varphi_i  \in GL(V_{\lambda_i})$ \\ such that $\varphi_i$ commutes with every element of $\rho_i(H_i)$.
\end{center}
\end{theorem}
\noindent Recall that integer-like means that the characteristic polynomial has coefficients in $\Z$ and constant term $\pm 1$. This condition is equivalent to only having eigenvalues which are algebraic integers. Note that the set of algebraic integers is closed under multiplication.

\begin{proof}
First assume that $\varphi \in \Aut(\lien_\G^\Q)$ is an Anosov automorphism commuting with every element of $\rho(H)$. As explained before we can assume that $\varphi \in L \le T$ or thus in particular that $\varphi(V) = V$. Take $g = p(\varphi) \in G$ the restriction of $\varphi$ to $V$ as introduced above. By taking a power of $\varphi$ we can also assume that $g \in G^0$ or thus $\varphi(V_\lambda) = V_\lambda$ for every $\lambda \in \Lambda$. So by taking $\varphi_i$ the restriction of $\varphi$ to the subspaces $V_{\lambda_i}$, we find in particular that $\rho_i(h) \circ \varphi_i = \varphi_i \circ \rho_i(h)$ for every $h \in H_i$, giving the first condition on the $\varphi_i$. 

For the second condition, we know that $\varphi$ is hyperbolic so we only have to check it for $\lambda_i$ with $c(\lambda_i) = 2$. Assume that $\mu_1, \mu_2$ are two eigenvalues of $\varphi_i$, then we have to show that $\vert \mu_1 \mu_2 \vert \neq 1$. We can assume that $\mu_1 \neq \mu_2$ since otherways there is nothing to prove. There are two possibilities, either $\lambda_i$ is a complete graph or there exists $\lambda^\prime \in \kappa_i$ such that $\lambda_i$ is connected to $\lambda^\prime$. 

In the first case, we have that $\mu_1 \mu_2$ is an eigenvalue of $\varphi$ by Remark \ref{eigenvalues}. In the second case, take $h \in H$ such that $h \cdot \lambda_i = \lambda^\prime$. Because $h(V_{\lambda_i}) = V_{\lambda^\prime}$ and $\varphi \comp h = h \comp \varphi$, we have that $\restr{\varphi}{V_{\lambda^\prime}} = h \comp \varphi_i \comp h^{-1}$ and thus $\restr{\varphi}{V_{\lambda^\prime}}$ has the same eigenvalues as $\varphi_i$. In particular, $\mu_1 \mu_2$ is again an eigenvalue of $\varphi$ by Remark \ref{eigenvalues}. So in both cases the statement follows from the hyperbolicity of $\varphi$, finishing the first implication of the theorem.

For the other implication, assume that $\varphi_i$ exists as in the theorem. Note that the eigenvalues of $\varphi_i$ are algebraic integers. We now construct a hyperbolic element $\varphi \in \Aut(\lien_\G^\Q)$ commuting with every element in $\rho_i(H_i)$. By taking some power of the $\varphi_i$, we can assume that for all eigenvalues $\mu_i, \mu_j$ of $\varphi_i$ and $\varphi_j$ respectively with $i\neq j$, we have that $\vert \mu_i \mu_j \vert \neq 1$.  Every $\lambda \in \Lambda$ lies in some orbit $\kappa_i$ of the action of $H$ on $\Lambda$ and hence is of the form $\pi(h)(\lambda_i)$ for some $i \in \{1,\ldots, k\}$, meaning that $\rho(h)(V_{\lambda_i}) = V_{\lambda}$. 
Now define $g_\lambda: V_\lambda \to V_\lambda$ as $g_\lambda =  \rho(h) \circ \varphi_i \circ \rho(h)^{-1}$. 

We claim that the definition of $g_\lambda$ does not depend on the choice of $h$ such that $\pi(h)(\lambda_i) = \lambda$. Indeed, assume that $h_1, h_2 \in H$ with $\pi(h_1)(\lambda_i) = \pi(h_2)(\lambda_i) = \lambda$, then $h_2^{-1} h_1 \in H_i$ by definition. By assumption $\varphi_i$ commutes with $\rho(h_2^{-1}h_1)$, so $\rho(h_2^{-1}h_1)\varphi_i = \varphi_i \rho(h_2^{-1}h_1)$ on $V_{\lambda_i}$ which implies the claim.

Take $g: V \to V$ as the direct sum of all $g_\lambda: V_\lambda \to V_\lambda$. It is clear that $g \in G$ and hence we can consider $\varphi \in \Aut(\lien_\G^\Q)$ with $p(\varphi) = g$. Because of the assumption on the eigenvalues of the $\varphi_i$ and the fact that $\varphi_i$ is $c(\lambda_i)$-hyperbolic, we will show that $\varphi$ is hyperbolic. Indeed, the eigenvalues of $g$ are also eigenvalues of some $\varphi_i$, since it is equal to a conjugate of $\varphi_i$ on $V_\lambda$ for every $\lambda \in \Lambda$. The eigenvalues of $\varphi$ are hence eigenvalue of some $\varphi_i$ or a product of two eigenvalues $\mu_{\lambda,i}\mu_{\lambda^\prime,j}$ for eigenvalues $\mu_{\lambda,i}$ and $\mu_{\lambda^\prime,j}$ of $g_\lambda$ and $g_{\lambda^\prime}$ with $\lambda$ and $\lambda^\prime$ connected. If $\lambda$ and $\lambda^\prime$ lie in the same orbit, then we know that $c(\lambda) = 2$ and thus that the eigenvalue is different from $1$ in absolute value. If $\lambda$ and $\lambda^\prime$ lie in a different orbit, then they correspond to different $\varphi_i$ and $\varphi_j$, hence the eigenvalue is different from $1$ in absolute value by the extra condition mentioned before. We conclude that $\varphi$ is a hyperbolic automorphism. In this proof we also showed that all the eigenvalues of $\varphi$ are algebraic integers, therefore the automorphism $\varphi$ is integer-like.

We are left to check that $\varphi$ commutes with the finite group $\rho(H)$. It suffices to check that $g$ commutes with every element $p(\rho(h)) \in p(\rho(H))$ on every subspace $V_\lambda$. Take $\lambda_i$ such that $\lambda \in\kappa_i$ and consider $\tilde{h} \in H$ such that $\rho(\tilde{h})(V_{\lambda_i}) = V_\lambda$. For every $x \in V_\lambda$, we get that $$g(x) = g_\lambda(x) =  \rho(\tilde{h}) \circ \varphi_i \circ \rho(\tilde{h}^{-1})(x).$$ Now $\rho(h \circ \tilde{h})$ maps $V_{\lambda_i}$ to $\rho(h)(V_\lambda)$, so the map $g_{\pi(h)(\lambda)}$ is given by $$\rho(h \circ \tilde{h}) \circ \varphi_i \circ \rho\left(( h \circ \tilde{h})^{-1}\right) (y) = \rho(h) \circ g_\lambda \circ \rho(h^{-1})(y)$$ for every $y \in \rho(h)(V_\lambda)$. The statement follows and this finishes the proof.
\end{proof}

In order to achieve a more workable condition, we recall the following result of \cite{dd13-1}. 

\begin{theorem} \label{vert}
	Let $\rho: G \to \GL(n,\Q)$ be a $\Q$-irreducible representation. Then there exists a $c$-hyperbolic, integer-like matrix $C\in \GL_{mn}(\Q)$ 
	which commutes with 
	$m \rho=\underbrace{\rho\oplus \rho \oplus \cdots \oplus \rho}_{m\ {\rm times}}$ 
	if and only if $\rho$ splits in strictly more than $\frac{c}{m}$ components when seen as a representation over $\R$.
\end{theorem}

We refer to \cite{dv08-1} for more details. Combined with Theorem \ref{vert}, this gives us the following result.

\begin{theorem}
	\label{maintheorem}
	Let $\G$ be a graph and $\lien^\Q_\G$ be the rational Lie algebra associated to this graph. Let $\rho: H \to \Aut(\lien^\Q_\G)$ be a representation of a finite group with $\rho_i: H_i \to \GL(V_{\lambda_i})$ the representations as introduced before. Then the following are equivalent.
	
	\begin{center}
		There exists an Anosov automorphism $\varphi \in \Aut(\lien^\Q_\G)$ commuting with every element of $\rho(H)$. \\
		$\Updownarrow$ \\
For every $1 \leq i \leq k$, every $\Q$-irreducible component of $\rho_i$ that occurs with multiplicity $m$\\ splits in strictly more than $\frac{c(\lambda_i)}{m}$ components over $\R$. 
	\end{center}
\end{theorem}
For our purposes, we can only have $c(\lambda_i) \in \{1, 2\}$. In the case $c(\lambda_i) = 1$, the condition states that every $\Q$-irreducible component of $\rho_i$ which is also irreducible as a representation over $\R$ must occur at least twice in $\rho_i$. If $c(\lambda_i) = 2$, then the condition states that every $\Q$-irreducible component which is irreducible over $\R$ must occur at least three times, and if a $\Q$-irruducible component splits in exactly $2$ components over $\R$ it must occur at least twice. The $\Q$-irreducible components which split in more components over $\R$ automatically satisfy the condition.

The trivial representation of dimension $1$ is both irreducible over $\Q$ and over $\R$. In particular, if every $\rho_i$ is the trivial representation, we get the following consequence.

\begin{corollary}\label{MainCorollary}
Let $\G$ be a graph and $\lien^\Q_\G$ be the rational Lie algebra associated to this graph. Let $\rho: H \to \Aut(\lien^\Q_\G)$ be a representation of a finite group such that the representation $\rho_i: H_i \to \GL(V_{\lambda_i})$ introduced above are trivial. Then there exists an Anosov automorphism $\varphi \in \Aut(\lien^\Q_\G)$ commuting with every element of $\rho(H)$ if and only if $\dim(V_{\lambda_i}) \geq 2$ if $c(\lambda_i) = 1$ and $\dim(V_{\lambda_i}) \geq 3$ otherwise.
	\end{corollary}
So if $H$ is trivial, we recover the main result of \cite{dm05-1}.
\section{Applications}
\label{sec:examples}

In this section, we apply the previous results to construct new examples of Anosov diffeomorphisms on infra-nilmanifolds. First we recall some machinery for constructing almost-Bieberbach groups given a rational Lie algebra $\lien^\Q$ and a faithful representation $\rho: H \to \Aut(\lien^\Q)$. Next we apply this to give concrete examples of graphs leading to infra-nilmanifolds admitting an Anosov diffeomorphism.

\paragraph{\textbf{Constructing almost-Bieberbach groups with given holonomy representation}}

An al\-most-\-Bie\-ber\-bach group $\Gamma \le \Aff(N)$ induces a torsion-free radicable nilpotent group $M^\Q \le N$ and a faithful representation of a finite group $\rho: H \to \Aut(M^\Q)$. But vice versa, given a torsion-free radicable nilpotent group $M^\Q$ and a faithful representation $\rho: H \to \Aut(M^\Q)$ of a finite group, it is not always easy to check whether $H$ corresponds to the rational holonomy representation of an almost-Bierberbach group $\Gamma$. The hard condition to achieve is the torsion-freeness of the group $\Gamma$. 

In some special cases, this problem is completely solved though. The case where $H$ is the trivial group follows immediately from the fundamental work of Mal'cev \cite{malc51-1}, since the group $M^\Q$ is always the radicable hull of some lattice $M \le N$ in a $1$-connected nilpotent Lie group $N$. The next interesting case is for cyclic groups $H$, which was treated in \cite{dv11-1} although it was not explicitly stated as a separate theorem therein. The proof of \cite[Proposition 5.2.]{dv11-1} contains the necessary arguments and for the convenience of the reader, we sketch the proof. 

\begin{theorem}
\label{theorem:eigenvalue1}
Let $\rho: H \to \Aut(M^\Q)$ be a faithful representation of a finite cyclic group $H$ on a torsion-free radicable nilpotent group $M^\Q$, where $\rho(H)$ is generated by the element $\phi \in H$. The following are equivalent:
\begin{center}
The representation $\rho$ is a rational holonomy representation of an almost-Bieberbach group $\Gamma$.  \\
	$\Updownarrow$ \\
The automorphism $\phi$ has eigenvalue $1$.
\end{center}
\end{theorem}

\begin{proof}
One direction is immediate, namely if $\rho: H \to \Aut(M^\Q)$ is the rational holonomy representation of an almost-Bieberbach group $\Gamma$ then $\phi$ has eigenvalue $1$. Indeed, assume that $\Gamma \le M^\Q \rtimes H$ is an almost-Bieberbach group with rational holonomy representation $H$. By definition there exists $n \in M^\Q$ such that $\gamma = (n,\phi) \in \Gamma$. Write $m = \gamma^{\vert H \vert} \in M \setminus \{e\}$, then $m = \gamma m \gamma^{-1} = n \phi(m) n^{-1}$. So the automorphism $M^\Q \to M^\Q: x \mapsto n \phi(x) n^{-1}$ has eigenvalue $1$, hence also the automorphism $\phi$ has eigenvalue $1$, since both automorphism are conjugate.  

For the other direction, it is easy to find a finitely generated torsion-free nilpotent subgroup $M$ such that $M^\Q$ is the radicable hull of $M$ and $\phi(M) =M$. The automorphism $\phi$ induces automorphisms $\phi_i \in \GL(n_i,\Z)$ on the quotients $\faktor{M \cap \gamma_i(M^\Q)}{M \cap \gamma_{i+1}(M^\Q)} \approx \Z^{n_i}$. By \cite[Lemma 4.3.]{dv11-1} there exists $M^\prime \le M^\Q$ which contains $M$ as a subgroup of finite index such that $\phi(M^\prime) = M^\prime$ and the induced representations $\phi_i \in \GL(n_i,\Z)$ for $M^\prime$ are totally reducible. Since $\phi$ has eigenvalue $1$, there exists $i$ such that the induced representation $\phi_i$ also has eigenvalue $1$. Then \cite[Proposition 5.1.]{dv11-1} implies there exists a torsion-free extension $1 \to \Z^{n_i} \to E \to \Z_{\vert H \vert} \to 1$ which induces $\phi_i$ and hence \cite[Theorem 4.1.]{dv11-1} gives us the desired result.
\end{proof}

\noindent The extra ingredient used in the proof of \cite[Proposition 5.2.]{dv11-1} was showing that every automorphism of finite order on a free nilpotent Lie algebra of nilpotency class $c \geq 2$ has $1$ as eigenvalue.
For graphs on the other hand, every graph automorphism induces an automorphism of the Lie algebra which automatically has eigenvalue $1$, leading to the following result.

\begin{corollary}\label{cor:graph}
Let $\G$ be a graph and $\phi \in \Aut(\G)$ be a graph automorphism. The natural representation $\langle \phi \rangle \to \Aut(\lien^\Q_\G)$ is the rational holonomy representation of an almost-Bieberbach group.
\end{corollary}
\begin{proof}
Consider the vector $v = \displaystyle \sum_{\alpha \in S} \alpha$ with $S$ the set of vertices of the graph $\G$. Since $\phi$ is an automorphism of the graph, the vector $v$ is an eigenvector for eigenvalue $1$ of the corresponding Lie algebra automorphism. The statement now follows immediately from Theorem \ref{theorem:eigenvalue1}.
\end{proof}
\noindent For the families of examples we will construct we will hence mainly focus on cyclic subgroups of graph automorphisms.

\paragraph{\textbf{Anosov diffeomorphisms for holonomy group $\Z_n$}}

We now give families of graphs $\G$ with  free actions of cyclic subgroups $H \subset \Aut(\G)$ by graph automorphism groups giving rise to infra-nilmanifolds admitting Anosov diffeomorphisms. The general strategy for constructing such examples is the following. We start with a graph $\overline{\G}$, with a vertex set $\Lambda$, on which a finite group $H$ acts freely, for example the Cayley graph of $H$ for some finite generating set. We realize the graph $\overline{G}$ as a quotient graph of a simple graph $\G$ satisfying the following three conditions:
	\begin{enumerate}
		\item the action of $H$ is induced by an action of $H$ via graph automorphisms on $\G$,
		\item each coherent class $\lambda \in \Lambda$ is of size at least 2, 
		\item the size of $\lambda$ is at least 3 whenever $\lambda \mu$ is an edge in $\overline{\G}$ for some $\mu$ in the orbit of $\lambda$.
	\end{enumerate}
	Using Corollary \ref{MainCorollary}, if there exists an infra-nilmanifold associated to the graph $\G$ corresponding to the action of $H$, it admits an Anosov diffeomorphism. Below we present examples for the cyclic groups $\Z_n$ and because of Corollary \ref{cor:graph}, these always correspond to an infra-nilmanifold. 

\begin{example}[Holonomy $\Z_2$]\label{FamilyI}
Consider a graph $\G$  with $2 m$ coherent classes, $m \geq 1$ where each coherent class is of the size $\ell$ where $\ell \geq 3$ and whose quotient graph is given by \[ \begin{tikzpicture}[font=\small,baseline=-4]
\node[vertex,label={below:$\lambda_1$}] (1) at (0,0) {};
\node[vertex,label={below:$\lambda_3$}] (2) at (1, 0) {};

\node[vertex,label={center:$$}] (3) at (1,0) {};
\node[circle,fill,scale=0,label={center:$\cdots$}] (4) at (1.5,0) {};

\node[vertex,label={below:$\lambda_{2m-1}$}] (5) at (2,0) {};
\node[vertex,label={below:$\lambda_{2m}$}] (6) at (3, 0) {};

\node[vertex,label={center:$$}] (8) at (3,0) {};
\node[circle,fill,scale=0,label={center:$\cdots$}] (7) at (3.5,0) {};

\node[vertex,label={below:$\lambda_4$}] (9) at (4,0) {};
\node[vertex,label={below:$\lambda_2$}] (10) at (5, 0) {};

\path[-]
 (1) edge (2)
 (2) edge  (3)
 (5) edge (6)
  (9) edge  (10);
\end{tikzpicture}\]

For $m = 1$, we get a complete bipartite graphs $K_{\ell, \ell}$. Consider an automorphism $\phi$ of the graph $\G$ that is of order 2 and  satisfies $\phi(\lambda_i) = \lambda_{i+1}$ if $i$ is odd.  Let $H$ denote a subgroup of $\Aut(\G)$ generated by $\phi$ so that $H \cong \mathbb Z_2$.  Then the stabilizers $H_i$ are all trivial and $\dim(V_{\lambda_i}) = \ell \geq 3$. Hence by Corollaries \ref{MainCorollary} and \ref{cor:graph}, the associated  infra-nilmanifold with holonomy group $\mathbb Z_2$ admits an Anosov diffeomorphism.

We can generalize the above family by varying the sizes of coherent classes and keeping the  size of $\lambda_i$ and $\lambda_{i+1}$ the same, for $i$ odd. More precisely, we assume that the size of $\lambda_i$ and $\lambda_{i+1}$ is $\ell_{\frac{i+1}{2}}$  for all odd $i$, where $\ell_{\frac{i+1}{2}} \geq 2$ for  $i \neq 2m-1$ and $\ell_m \geq 3$. Note that we need $\ell_m \geq 3$ as $c(\lambda_{2m-1}) = 2$ (see Corollary \ref{MainCorollary}). Using the same graph automorphism $\phi$ as above, we will get the associated infra-nilmanifolds  admitting Anosov diffeomorphisms with holonomy group $\mathbb Z_2$. We note that the  number of vertices in $\G$ is $\displaystyle \sum_{i=1}^m 2 \ell_i$ and the number edges in $\G$ is $\displaystyle \sum_{i=1}^{m-1} 2 \ell_{i} \ell_{i+1} + \ell_m^2 $. Hence this family gives us examples of infra-nilmanifolds  of dimensions of the form  $\displaystyle \sum_{i=1}^m 2 \ell_i + \sum_{i=1}^{m-1} 2 \ell_{i} \ell_{i+1} + \ell_m^2$ where $m \geq 1$, $\ell_m \geq 3$ and $\ell_i \geq 2$ for $i \neq m$. In particular, for $m \geq 2$, by letting $\ell_m = 3$ and $\ell_i = 2$ for all $ 1 \leq i < m$, the dimension of the corresponding  infra-nilmanifold is $(4m+2) + (8m+5) = 12 m + 7$.  
If $m =1$, the dimension is $6+9= 15$. 

\end{example}

\begin{remark} One can modify the Family I in many different ways. We describe one such a modification.  Let $ m \geq 3,  \ell_1 = \ell_m = 3$ and  $\ell_i =2 $ for $2 \leq i \leq m-1$. The sizes of the coherent classes are the same as in the generalized family as described above, i.e., $\# \lambda_i = \# \lambda_{i+1} = \ell_{\frac{i+1}{2}}$  for all odd $i$.  We consider a graph $\G$ whose quotient graph is 

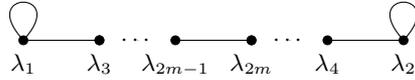
\[ \begin{tikzpicture}[font=\small,baseline=-4]
\node[vertex,label={below:$\lambda_1$}] (1) at (0,0) {};
\node[vertex,label={below:$\lambda_3$}] (2) at (1, 0) {};

\node[vertex,label={center:$$}] (3) at (1,0) {};
\node[circle,fill,scale=0,label={center:$\cdots$}] (4) at (1.5,0) {};

\node[vertex,label={below:$\lambda_{2m-1}$}] (5) at (2,0) {};
\node[vertex,label={below:$\lambda_{2m}$}] (6) at (3, 0) {};

\node[vertex,label={center:$$}] (8) at (3,0) {};
\node[circle,fill,scale=0,label={center:$\cdots$}] (7) at (3.5,0) {};

\node[vertex,label={below:$\lambda_4$}] (9) at (4,0) {};
\node[vertex,label={below:$\lambda_2$}] (10) at (5, 0) {};

\path[-]
(1) edge [loopup] (1)
 (1) edge (2)
 (2) edge  (3)
 (5) edge (6)
  (9) edge  (10)
  (10) edge [loopup] (10)
  ;
\end{tikzpicture}\]

In this case also by considering the same type of order 2 graph automorphism as in Example \ref{FamilyI}, we get a family of infra-nilmanifolds admitting Anosov diffeomorphisms with holonomy group $\mathbb Z_2$.  The number of vertices of $\G$ is then $4m+4$ and the number of edges is $8m+5+10$. Hence the dimension of the corresponding 2-step nilpotent Lie algebra is $12m + 19$. We observe that these dimensions are already covered by Example \ref{FamilyI}. However, this graph and a graph from the Family I are not isomorphic and hence the associated 2-step nilpotent Lie algebras are non isomorphic \cite{main15-1}. This gives us examples of non isomorphic infra-nilmanifolds admitting Anosov diffeomorphisms with holonomy group $\mathbb Z_2$ and with dimension of the form $12m + 19$ for $m \geq 3$.

 \end{remark}

\begin{example}[Holonomy $\Z_n$ with $n \geq 3$]\label{FamilyII} 

Let $n \geq 3$, $n \neq 4$. We consider a graph $\G$ with $n$ coherent classes $\{\lambda_1, \ldots, \lambda_{n}\}$ all of size 3 and whose quotient graph $\overline{\G}$  is a cycle graph  with $\lambda_i \lambda_{(i+1)\mod n}$ an edge in $\overline{\G}$ for all $i$, $1\leq i \leq n$.
 Let $\phi$ denote the graph automorphism of $\G$ of order $n$ satisfying $\phi(\lambda_i) = \lambda_{(i+1)\mod n}$ for all $i$ and $H$ be a cyclic subgroup of $\Aut(\G)$ generated by $\phi$ so that $H \cong \mathbb Z_{n}$. Then an infra-nilmanifold  associated to $\G$ admits an Anosov diffeomorphism with holonomy group $\mathbb Z_n$. The dimension of the infra-nilmanifold is $ 3n + 9 n = 12 n.$ This can be generalized by choosing the size of each coherent class to be $\ell \geq 3$ to get dimension $\ell n  + \ell^2 n$.

 For holonomy group $\mathbb Z_4$, we consider a graph  with $4$ coherent classes each of size $\ell \geq 3$ and whose quotient graph looks like
 
 
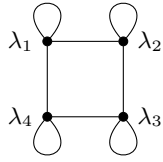
\[ \begin{tikzpicture}[font=\small,baseline=-4]
\node[vertex,label={left:$\lambda_1$}] (1) at (0,0) {};
\node[vertex,label={right:$\lambda_2$}] (2) at (1, 0) {};
 
\node[vertex,label={right:$\lambda_3$}] (3) at (1,-1) {};
\node[vertex,label={left:$\lambda_4$}] (4) at (0, -1) {};
\path[-]
(1) edge [loopup] (1)
(2) edge [loopup] (2)
(3) edge [loopdown] (3)
(4) edge [loopdown] (4)

 (1) edge (2)
 (2) edge  (3)
 (3) edge (4)
  (4) edge (1);
\end{tikzpicture}\] 

Here $H$ would be a subgroup of graph automorphism group generated by an order 4 rotation mapping $\lambda_i$ to $\lambda_{(i+1)\text{mod} 4}$. The corresponding infra-nilmanifold is of dimension $4 \ell + 2 \ell (\ell -1) + 4 \ell^2$. The minimum possible such a dimension  is  $60$ when $\ell = 3$.

\end{example}

\begin{remark}
In Family II described as above, we note that we had to deal with the $\mathbb Z_4$ case separately. This is because a  simple cycle  graph on 4 vertices  can not be realized as a quotient graph of a graph. Here we want to note that sufficient conditions for a given graph to be a quotient graph are given in  \cite[ Lemma 4.6]{MPS18}
\end{remark}

\section{Open questions}
\label{sec:openQ}

Although Theorem \ref{maintheorem} seems to give a full answer about the existence of Anosov diffeomorphisms on Lie groups associated to graphs, it only considers one type of uniform lattices of these Lie groups. However, if $N$ is a $1$-connected nilpotent Lie group with a uniform lattice $M \le N$, then the existence of an Anosov diffeomorphism on the nilmanifold $\faktor{N}{M}$ depends on the lattice $M$. For example, if $N_\G$ is the Lie group associated to the graph  $ \G $ given as follows
\[ \begin{tikzpicture}[font=\small,baseline=-4]
\node[vertex,label={left: $\alpha$}] (1) at (0,0) {};
\node[vertex,label={right:$\beta$}] (2) at (1, 0) {};
 
\node[vertex,label={right:$\delta$}] (3) at (1,-1) {};
\node[vertex,label={left:$\gamma$}] (4) at (0, -1) {};
\path[-]

 (1) edge (2)
 
 (3) edge (4)
 ;
\end{tikzpicture}\] 
\noindent i.e.~$N_\G$ is the direct sum of two real Heisenberg groups, then there exist uniform lattices $M_1, M_2 \le N_\G$ such that $\faktor{N_\G}{M_1}$ admits an Anosov diffeomorphism but $\faktor{N_\G}{M_2}$ does not.
Equivalently, the Lie algebra $\lien_\G^\R$ has two distinct rational forms where one has an Anosov automorphism and the other does not \cite{malf97-3}. Hence when studying Anosov diffeomorphisms, even for nilmanifolds associated to graphs, the choice of rational form is crucial. For free nilpotent Lie groups, there is only one posible uniform lattice up to commensurability, hence this problem did not occur in the work \cite{dv09-1}.

Note that in this paper we only considered infra-nilmanifolds associated to graphs, which is by definition equivalent to considering only the standard rational form $\lien_\G^\Q \subset \lien_\G^\R$ of Lie algebras associated to graphs. It is still an open problem to describe the other rational forms of $\lien_\G^\R$ and characterize the ones which admit an Anosov automorphism.

\begin{QN}
	\label{Q1}
Is there a description of all rational forms $\liem^\Q \subset \lien^\R_\G$ of real Lie algebras associated to graphs? Can we characterize the ones which admit an Anosov automorphism? 
\end{QN}
Even if $\lien_\G^\Q$ is not Anosov, their could be other rational forms which are Anosov, as the example $\G$ of the direct sum of two Heisenberg algebras shows. But if $\lien_\G^\Q$ is Anosov, all low-dimensional examples seem to imply the following conjecture.

\begin{Con}
Let $\lien_\G^\Q$ be a Lie algebra associated to a graph $\G$ admitting an Anosov automorphism. If  $\liem^\Q$ is any rational form of the real Lie algebra $\lien_\G^\R$, then $\liem^\Q$ is Anosov as well.
\end{Con}
If a full answer to the Question \ref{Q1} would be known, the natural follow-up question would be the generalization of Theorem \ref{maintheorem} to other rational forms of $\lien_\G^\R$.

\begin{QN}
Is there a characterization of the infra-nilmanifolds modeled on Lie groups associated to graphs which admit an Anosov diffeomorphism?
\end{QN}

\noindent Note that there is an important distinction between infra-nilmanifolds modeled on Lie groups associated to graphs and infra-nilmanifolds associated to graphs, since by Definition \ref{def-associated} we assume that the rational Lie algebra of the latter is equal to $\lien^\Q_\G$.

\bibliographystyle{alpha}
\bibliography{ref}

\end{document}